\documentclass{article}
\usepackage[english]{babel}
\usepackage[utf8]{inputenc}
\usepackage{amsmath}
\usepackage{graphicx}
\usepackage{amsfonts}
\usepackage{enumitem}
\usepackage{amssymb}
\usepackage[colorinlistoftodos]{todonotes}
\usepackage{geometry}
\usepackage{amsthm}

\newcommand{\ang}[1]{\left\langle #1 \right\rangle}

\newcommand{\paren}[1]{\left( #1 \right)}

\newcommand{\set}[1]{\left\{ #1 \right\}}
\newcommand{\setcond}[2]{\left\{ #1 \;\middle\vert\; #2 \right\}}

\newcommand{\CC}{\mathbb{C}}

\newcommand{\RR}{\mathbb{R}}

\newcommand{\NN}{\mathbb{N}}
\newcommand{\ZZ}{\mathbb{Z}}
\newcommand{\QQ}{\mathbb{Q}}
\newcommand{\cP}{\mathcal{P}}
\newcommand{\cQ}{\mathcal{Q}}
\newcommand{\cL}{\mathcal{L}}
\newcommand{\cM}{\mathcal{M}}
\newcommand{\cR}{\mathcal{R}}

\newtheorem{thm}{Theorem}[section]
\newtheorem{lem}[thm]{Lemma}
\newtheorem{cor}[thm]{Corollary}
\newtheorem{conj}[thm]{Conjecture}

\theoremstyle{definition}
\newtheorem{defn}[thm]{Definition}

\DeclareMathOperator{\GL}{GL}
\DeclareMathOperator{\Mat}{Mat}

\title{Sums of linear transformations}
\author{David Conlon\thanks{Department of Mathematics, Caltech, Pasadena, CA 91125, USA. Email: {\tt dconlon@caltech.edu}. Research supported by NSF Awards DMS-2054452 and DMS-2348859.} \and Jeck Lim\thanks{Department of Mathematics, Caltech, Pasadena, CA 91125, USA. Email: {\tt jlim@caltech.edu}. Research partially supported by an NUS Overseas Graduate Scholarship.}}
\date{}

\begin{document}
\maketitle

\begin{abstract}
We show that if $\cL_1$ and $\cL_2$ are linear transformations from $\ZZ^d$ to $\ZZ^d$ satisfying certain mild conditions, then, for any finite subset $A$ of $\ZZ^d$,
$$|\cL_1 A+\cL_2 A|\geq \paren{|\det(\cL_1)|^{1/d}+|\det(\cL_2)|^{1/d}}^d|A|- o(|A|).$$ 
This result corrects and confirms the two-summand case of a conjecture of Bukh and is best possible up to the lower-order term for certain choices of $\cL_1$ and $\cL_2$. As an application, we prove a lower bound for $|A  + \lambda \cdot A|$ when $A$ is a finite set of real numbers and $\lambda$ is an algebraic number. In particular, when $\lambda$ is of the form $(p/q)^{1/d}$ for some $p, q, d \in \NN$, each taken as small as possible for such a representation, we show that
$$|A + \lambda \cdot A| \geq (p^{1/d} + q^{1/d})^d |A| - o(|A|).$$
This is again best possible up to the lower-order term and extends a recent result of Krachun and Petrov which treated the case $\lambda = \sqrt{2}$.
\end{abstract}

\section{Introduction}

For any subset $A$ of a commutative ring $R$ (or, more generally, an $R$-module $M$) and any elements $\lambda_1, \dots, \lambda_k$ of $R$, let
\[\lambda_1 \cdot A + \cdots + \lambda_k \cdot A = \{\lambda_1 a_1 + \cdots + \lambda_k a_k : a_1, \dots, a_k \in A\}.\]
Such sums of dilates have attracted considerable attention in recent years, with the basic problem asking for an estimate on the minimum size of $|\lambda_1 \cdot A + \cdots + \lambda_k \cdot A|$ given $|A|$. Over the integers, this problem was essentially solved by Bukh~\cite{B08}, who showed that if $\lambda_1, \dots, \lambda_k$ are coprime integers, then, for any finite set of integers $A$,
\[|\lambda_1 \cdot A + \cdots + \lambda_k \cdot A| \geq (|\lambda_1| + \cdots + |\lambda_k|)|A| - o(|A|),\]
which is best possible up to the lower-order term. This result was later tightened by Balog and Shakan~\cite{BS14} when $k = 2$ and then Shakan~\cite{S16} in the general case, improving the $o(|A|)$ term to a constant depending only on $\lambda_1, \dots, \lambda_k$ (see also~\cite{CHS09, CSV10, DCS14, HR11, L13} for some earlier work on specific cases). 

Our principal concern here will be with generalisations of these results to higher dimensions. One possible direction is to again look at sums of dilates (see, for example,~\cite{BS15, CL22, H21, M19, M21, M22}). However, we will be concerned with a generalisation of a different kind, encapsulated in the following conjecture of Bukh. This conjecture first appeared on Bukh's webpage, but has since been reiterated by several other authors~\cite{KP20, M19, S16}.

\begin{conj} \label{conj:bukh}
Suppose that $\cL_1,\ldots, \cL_k \in \Mat_d(\ZZ)$ have no common non-trivial invariant subspace and $\cL_1\ZZ^d+\cdots +\cL_k\ZZ^d=\ZZ^d$. Then, for any finite subset $A$ of $\ZZ^d$,
$$|\cL_1 A+\cdots +\cL_k A|\geq (|\det(\cL_1)|^{1/d}+\cdots +|\det(\cL_k)|^{1/d})^d|A| - o(|A|).$$
\end{conj}

The intuition behind this conjecture comes from the Brunn--Minkowski inequality (see, for example,~\cite{Gar02}). This classic inequality states that if $A$ and $B$ are two non-empty compact subsets of $\RR^d$, then
$$\mu(A+B)^{1/d}\geq \mu(A)^{1/d}+\mu(B)^{1/d},$$
where $\mu$ is the Lebesgue measure on $\RR^d$. Since $\mu(\cL A) = |\det(\cL)| \mu(A)$ for any $\cL \in \Mat_d(\RR)$ and any measurable subset $A$ of $\RR^d$, we may conclude that, for any $\cL_1, \cL_2 \in \Mat_d(\RR)$, 
$$\mu(\cL_1 A+\cL_2 A) \geq (\mu(\cL_1 A)^{1/d} + \mu(\cL_2 A)^{1/d})^d \geq (|\det(\cL_1)|^{1/d}+|\det(\cL_2)|^{1/d})^d \mu(A).$$
Moreover, the analogous statement holds for the sum of more transformations by a simple induction. Conjecture~\ref{conj:bukh} is then the statement that, under appropriate technical conditions, a discrete analogue of this result should hold, possibly with some correction term to deal with boundary effects.

The first result towards this conjecture was given by Mudgal~\cite{M19}, who showed that if $\cL \in GL_2(\RR)$ has no real eigenvalues, then $|A + \cL A| \geq 4|A| - o(|A|)$ for any finite subset $A$ of $\RR^2$. In particular, this confirms Conjecture~\ref{conj:bukh} when $k = d = 2$, $\cL_1$ is the identity and $|\det(\cL_2)| = 1$.\footnote{There is a caveat here, which is that Mudgal’s result, which applies to arbitrary subsets of $\RR^2$, requires that $\cL$ have no non-trivial invariant subspace over $\RR$. Our interpretation of Bukh’s conjecture, which concerns subsets of $\ZZ^d$  (or $\QQ^d$), is that there is instead no non-trivial invariant subspace over $\QQ$.} Surprisingly, despite this success, it turns out that Bukh's conjecture is not quite correct and both conditions, that $\cL_1,\ldots, \cL_k$ have no common non-trivial invariant subspace and that $\cL_1\ZZ^d+\cdots +\cL_k\ZZ^d=\ZZ^d$, need modification. 

The first condition, that $\cL_1,\ldots, \cL_k$ have no common non-trivial invariant subspace is clearly necessary, since otherwise, for subsets $A$ of such a common invariant subspace, the problem reduces to one of lower dimension. However, this is not the only case where the problem can reduce to one of lower dimension. For instance, a simple concrete example where this can happen is when $d=k=2$ and both $\cL_1$ and $\cL_2$ are anti-clockwise rotations about the origin by $\pi/2$. Indeed, even though $|\det(\cL_1)| = |\det(\cL_2)| = 1$, so that the conjecture predicts that $|\cL_1 A+\cL_2 A|\geq 4|A| - o(|A|)$, we only have $|\cL_1 A+\cL_2 A|=2|A|-1$ when $A=\setcond{(0,x)}{x\in [n]}$. In order to rule out such examples, we update Bukh's condition as follows.

\begin{defn}
We say that $\cL_1,\ldots,\cL_k\in \Mat_d(\ZZ)$ are \textit{irreducible} if there are no non-trivial subspaces $U$, $V$ of $\QQ^d$ of the same dimension such that $\cL_i U\subseteq V$ for all $i$.
\end{defn}

To reiterate the point, this condition is clearly necessary, since otherwise we may restrict $A$ and the $\cL_i$ to $U$, again reducing the problem to one of lower dimension.

Consider now the transformations
$$\cL_1=\begin{pmatrix}
2 & 0\\
0 & 1
\end{pmatrix}, \quad \cL_2=\begin{pmatrix}
0 & -1\\
1 & 0
\end{pmatrix}\begin{pmatrix}
2 & 0\\
0 & 1
\end{pmatrix}=\begin{pmatrix}
0 & -1\\
2 & 0
\end{pmatrix}.$$
It is easily checked that $\cL_1$ and $\cL_2$ are irreducible and that $\cL_1\ZZ^2+\cL_2\ZZ^2=\ZZ^2$. However, the set $A=\setcond{(x,2y)}{x,y\in [n]}$ has 
$|A|=n^2$ and $|\cL_1 A+\cL_2 A|=(2n-1)^2\sim 4|A|$,
giving another counterexample to Conjecture~\ref{conj:bukh}, which predicts that $|\cL_1 A+\cL_2 A| \geq 8|A| - o(|A|)$. 

The issue here is that $\cL_1$ and $\cL_2$ have a ``common right factor'' with determinant of absolute value $>1$. On the other hand, Bukh's condition that $\cL_1\ZZ^d+\cdots +\cL_k\ZZ^d=\ZZ^d$ is equivalent to $\cL_1, \dots, \cL_k$ not having a ``common left factor'' with determinant of absolute value $>1$. Indeed, if $\cL_1\ZZ^d+\cdots +\cL_k\ZZ^d=L\subsetneq\ZZ^d$, then there is some $\cP\in\Mat_d(\ZZ)$ with determinant of absolute value $>1$ such that $\cP\ZZ^d\supseteq L$, which implies that $\cP^{-1}\cL_i\ZZ^d\subseteq \ZZ^d$ and so $\cP^{-1}\cL_i\in\Mat_d(\ZZ)$ for all $i$. Conversely, if there is some $\cP\in\Mat_d(\ZZ)$ with determinant of absolute value $>1$ such that $\cP^{-1}\cL_i\in\Mat_d(\ZZ)$ for all $i$, then $\cL_1\ZZ^d+\cdots +\cL_k\ZZ^d \subseteq \cP \ZZ^d\subsetneq \ZZ^d$.  Our second condition incorporates and generalises both of these possibilities. 

\begin{defn}
We say that $\cL_1,\ldots,\cL_k\in\Mat_d(\ZZ)$ are \textit{coprime} if there are no 
$\cP,\cQ\in\GL_d(\QQ)$ with $0<|\det(\cP)\det(\cQ)|<1$ such that
$$\cP\cL_1\cQ,\cP\cL_2\cQ,\ldots, \cP\cL_k\cQ\in \Mat_d(\ZZ).$$
In particular, $\cL_1\ZZ^d+\cdots+\cL_k\ZZ^d=\ZZ^d$.
\end{defn}

To see that this condition is also necessary, observe that, for any $A\subset\QQ^d$, if we let $A'=\cQ^{-1}A$, then $|A'|=|A|$ and $|\cL_1 A+\cdots+\cL_k A|=|\cP\cL_1\cQ A'+\cdots +\cP\cL_k\cQ A'|$. But the transformations $\cP\cL_i\cQ$ have smaller determinants, suggesting that the lower bound should instead be phrased in terms of these determinants.

Taking all these observations into account, we arrive at the following modified version of Bukh's conjecture.

\begin{conj} \label{conj:bukh2}
Suppose that $\cL_1,\ldots,\cL_k\in \Mat_d(\ZZ)$ are irreducible and coprime. Then, for any finite subset $A$ of $\ZZ^d$,
$$|\cL_1 A+\cdots +\cL_k A|\geq \paren{|\det(\cL_1)|^{1/d}+\cdots +|\det(\cL_k)|^{1/d}}^d|A| - o(|A|).$$
\end{conj}

Our main result is a proof of this modified conjecture for $k=2$ and any $d$ in the following strong form. We note that this result is best possible up to the lower-order term in certain cases, for instance, when $d=2$, $\cL_1$ is the identity and $\cL_2 \in \Mat_2(\ZZ)$ is a dilate of a rotation about the origin through an angle which is not an integer multiple of $\pi$.

\begin{thm} \label{thm:main}
Suppose that $\cL_1,\cL_2\in \Mat_d(\ZZ)$ are irreducible and coprime. Then there are constants $D$, $\sigma>0$ such that, for any finite subset $A$ of $\ZZ^d$, 
$$|\cL_1 A+\cL_2 A|\geq \paren{|\det(\cL_1)|^{1/d}+|\det(\cL_2)|^{1/d}}^d|A|-D|A|^{1-\sigma}.$$
\end{thm}

The proof of this result has two main steps. First, in Section~\ref{sec:BM}, we use compression methods to prove a certain discrete version of the Brunn--Minkowski inequality. The core of the proof is then a bootstrapping argument that starts with a trivial bound and repeatedly improves it using our Brunn--Minkowski inequality, ultimately approaching the estimate stated in Theorem~\ref{thm:main}. 
Because the details of this second step are rather easier to digest when $\cL_1$ is the identity map, we will, in Section~\ref{sec:iden}, first prove Theorem~\ref{thm:main} in this special case. We then prove the full result in Section~\ref{sec:main}. 

As an application of Theorem~\ref{thm:main}, we make progress on a question raised recently by Shakan~\cite{S16} and by Krachun and Petrov~\cite{KP20}, namely, for a fixed real algebraic number $\lambda$ and a finite subset $A$ of $\RR$, how large is $|A+\lambda\cdot A|$ in terms of $|A|$? The analogous question when $\lambda$ is transcendental has received considerable attention~\cite{KL06, S08, S12, Sch11}, culminating in a recent result of the authors~\cite{CL23}, drawing on some of the methods of this paper, saying that $|A+\lambda\cdot A| \geq e^{c \sqrt{\log |A|}}|A|$ for some $c > 0$, which is best possible up to the value of $c$.


For algebraic $\lambda$, there is a general result due to Chen and Fang~\cite{CF18}, itself improving an earlier result of Breuillard and Green~\cite{BG13}, saying that, for any fixed $\lambda \geq 1$, $|A + \lambda \cdot A| \geq (1 + \lambda - o(1))|A|$ 
holds for all finite subsets $A$ of $\RR$. This is best possible when $\lambda$ is an integer, but can be quite slack in other cases. For instance, when $\lambda = p/q$ with $p$ and $q$ coprime, then the result of Balog and Shakan~\cite{BS14} that $|p \cdot A + q \cdot A| \geq (p+q)|A|  - C_{p,q}$ for all finite subsets $A$ of $\ZZ$ easily implies that 
$$|A + \tfrac pq \cdot A| \geq (p+q)|A|  - C_{p,q}$$ 
for all finite subsets $A$ of $\RR$. 

In their paper, Krachun and Petrov~\cite{KP20} studied the case where $\lambda = \sqrt{2}$, showing that 
$$|A + \sqrt{2} \cdot A| \geq (1 + \sqrt{2})^2|A|  - o(|A|),$$ 
which is best possible up to the lower-order term, as may be seen by considering the set $A  = \{x + y \sqrt{2}: 0 \leq x < M, 0  \leq y < N\}$ with the ratio $M/N$ approaching $\sqrt{2}$. They also formulated a conjecture for a general real algebraic $\lambda$. Indeed, if $f(x) \in \mathbb{Z}[x]$ is the minimal polynomial of $\lambda$, assumed to have coprime coefficients, and $f(x) = \prod_{i=1}^d (a_i x + b_i)$ is a full complex factorisation of $f$, let $H(\lambda) = \prod_{i=1}^d (|a_i| + |b_i|)$. Then the conjecture of Krachun and Petrov, for which they prove the upper bound, is as follows. 

\begin{conj}[Krachun--Petrov~\cite{KP20}] \label{conj:KP}
For any real algebraic number $\lambda$,
\[\lim_{n \rightarrow \infty} \min_{A \subset \RR, |A| = n} \frac{|A + \lambda \cdot A|}{|A|} = H(\lambda).\]
\end{conj}

While we fall short of proving this conjecture, Theorem~\ref{thm:main} does allow us to prove the following result. 

\begin{thm} \label{thm:alg}
Suppose that $\lambda\in\RR$ is an algebraic number with minimal polynomial $p(x)=a_dx^d+\cdots+a_0\in\ZZ[x]$, where all the $a_i$ are coprime. Then there are constants $D$, $\sigma>0$ such that 
$$|A+\lambda\cdot A|\geq (|a_d|^{1/d}+|a_0|^{1/d})^d|A| - D|A|^{1-\sigma}$$
holds for all finite subsets $A$ of $\RR$. 
\end{thm}

In particular, when $\lambda$ is of the form $(p/q)^{1/d}$ for some $p, q, d \in \NN$, each taken as small as possible for such a representation, the minimal polynomial for $\lambda$ is $f(x) = q x^d - p$, so that there are constants $D$, $\sigma>0$ such that 
$$|A+\lambda\cdot A| \geq (p^{1/d} + q^{1/d})^d |A| - D|A|^{1-\sigma}.$$
However, it is also easy to see that $H((p/q)^{1/d}) = (p^{1/d} + q^{1/d})^d$, so that we have confirmed the Krachun--Petrov conjecture in this case. 

\begin{cor}
For any $\lambda$ of the form $(p/q)^{1/d}$ for some $p, q, d \in \NN$, each taken as small as possible for such a representation, there are constants $D$, $\sigma>0$ such that 
$$|A+\lambda\cdot A| \geq (p^{1/d} + q^{1/d})^d |A| - D|A|^{1-\sigma}$$
holds for all finite subsets $A$ of $\RR$. Moreover, this is best possible up to the lower-order term.
\end{cor}

For further details, we refer the reader to Section~\ref{sec:alg}.

\section{A discrete Brunn--Minkowski inequality} \label{sec:BM}

In this section, we begin our proof of  Theorem~\ref{thm:main} by using compression arguments to establish the following discrete analogue of the Brunn--Minkowski theorem. We refer the reader to~\cite{BL96, GG01, GT06}
for a selection of results in a similar vein.

\begin{lem} \label{lem:discbm}
Fix a basis $\set{b_1,\ldots,b_d}$ of $\RR^d$. For each  $I\subseteq [d]$, let $p_I:\RR^d\to \RR^I$ be the projection onto the span of $\set{b_i}_{i\in I}$ along the given basis. Then, for any finite subsets  $A,B$  of $\RR^d$,
$$|A+B|\geq (|A|^{1/d}+|B|^{1/d})^d-\sum_{I\subsetneq [d]}|p_I(A+B)|.$$
\end{lem}

\begin{proof}
By applying a suitable linear transformation, we may assume that the basis is the standard one. Let $p_i=p_{[d]\setminus\{i\}}:\RR^d\to \RR^{d-1}$ be the linear map that removes the $i$th coordinate. 

We define $i$-compressions for $i=1,\ldots,d$ as follows. For a set $A\subset \RR^d$ and a point $x\in p_i(A)$, let $A_x=p_i^{-1}(x)$. Define the $i$-compression of $A$ to be the set $A'$ such that $p_i(A')=p_i(A)$ and, for each $x\in p_i(A)$, the $i$th coordinates of $A_x'$ are $0,1,\ldots,|A_x|-1$. Note that $|A'|=|A|$, so an $i$-compression does not alter the size of the set. 

Suppose that $A'$ and $B'$ are the $i$-compressions of $A$ and $B$. We will now  show that $|A'+B'|\leq |A+B|$ and, more generally, that 
$|p_S(A'+B')|\leq |p_S(A+B)|$ for any $S \subseteq [d]$. If $i\not\in S$, then $p_S(A'+B')=p_S(A+B)$. We may therefore assume that $i\in S$. Let $T=S\setminus \{i\}$. For a set $C$ and $x\in p_T(C)$, denote by $(p_S(C))_x$ the set $\setcond{y\in p_S(C)}{p_T(y)=x}$. Then, for any $z\in p_T(A'+B')=p_T(A')+p_T(B')$, there is some $x\in p_T(A')=p_T(A)$ and $y\in p_T(B')=p_T(B)$ such that $x+y=z$ and $|(p_S(A'+B'))_z|=|(p_S(A'))_x|+|(p_S(B'))_y| - 1$. Hence, 
$$|(p_S(A+B))_z|\geq |(p_S(A))_x|+|(p_S(B))_y|-1=|(p_S(A'))_x|+|(p_S(B'))_y|-1=|(p_S(A'+B'))_z|.$$
Taking the sum over all $z$, we have $|p_S(A'+B')|\leq |p_S(A+B)|$, as claimed. We therefore see that if the required inequality holds for the $i$-compressions of $A$ and $B$, then it also holds for the original sets.

By repeatedly taking $i$-compressions for $i=1,\ldots,d$, we may assume that $A,B\subset \ZZ_{\geq 0}^d$. We will say that $A$ is $i$-compressed if the $i$-compression of $A$ is $A$ itself and $A$ is compressed if it is $i$-compressed for all $i$. Now, by considering the sum of the coordinates of all the points of $A$ or $B$, we see that taking the $i$-compression strictly decreases these sums unless they are already $i$-compressed. Therefore, by repeatedly taking $i$-compressions for each $i$, we may assume that $A$ and $B$ are compressed. This means that for any points $(x_1,\ldots,x_d)\in A$ and $(y_1,\ldots,y_d)$ such that $0\leq y_i\leq x_i$ for all $i$,  $(y_1,\ldots,y_d)\in A$ and similarly for $B$. 

For a point $x=(x_1,\ldots,x_d)\in \ZZ^d$, let $C_x$ be the closed cube 
$\prod_{i=1}^d [x_i-1,x_i].$
Define 
$A^*=\bigcup_{x\in A} C_x,$
a compact set with $\mu(A^*)=|A|$, and define $B^*$ similarly. Then, by the Brunn--Minkowski inequality, we have
$$\mu(A^*+B^*)\geq (|A|^{1/d}+|B|^{1/d})^d.$$
We can write $A^*+B^*$ as the union of closed cubes
$$A^*+B^*=\bigcup_{x\in A+B+\{0,-1\}^d} C_x.$$
Since $A$ and $B$ are compressed, so is $A+B$. Using this fact, we can rewrite $A^*+B^*$ as a union of closed sets with disjoint interiors in the following way. For $S\subseteq [d]$, let $P_S$ be the set of points in $\ZZ^d$ such that $p_S(P_S)=p_S(A+B)$ and the coordinates outside of $S$ are all $-1$. Notice that the $P_S$ are pairwise disjoint for each $S\subseteq [d]$ and $P_S\subseteq A+B+\set{0,-1}^d$. Furthermore, for each $x\in A+B+\set{0,-1}^d$, let $S$ be the set of coordinates of $x$ which are not $-1$. Then $x\in P_S$, so that 
$$A^*+B^*=\bigcup_{S\subseteq [d]}\bigcup_{x\in P_S} C_x.$$
In particular, $\mu(A^*+B^*)=\sum_{S\subseteq [d]} |P_S|=\sum_{S\subseteq [d]} |p_S(A+B)|$. Hence, since $p_{[d]}(A+B)=A+B$, we have 
$$\mu(A^*+B^*)= |A+B|+\sum_{I\subsetneq [d]}|p_I(A+B)|$$
and the lemma follows.
\end{proof}

Our aim now is to apply this discrete Brunn--Minkowski inequality to prove an estimate that will play an important role in the bootstrap arguments of the next two sections. For this, we will need several additional ingredients, beginning with the following classical theorem of Freiman~\cite{Fr73} (see also~\cite{Bi99}) on subsets of small doubling in torsion-free abelian groups. Given such a group $G$, a proper progression $P$ of dimension $s$ and size $L$ is a  set of the form 
$$P=\setcond{v_0+u_1v_1+\cdots+u_sv_s}{0\leq u_i<L_i  \mbox { for } 1\leq i\leq s},$$
where $L_1L_2\cdots L_s=L$, $v_0,v_1,\ldots,v_s$ are elements of $G$ and all of the sums arising in the definition of $P$ are distinct. 

\begin{thm} \label{thm:gr}
For any $K > 0$, there exist constants $C_1$ and $C_2$ such that if $A$ is a subset of a torsion-free abelian group $G$ with $|A+A|\leq K|A|$, then $A$ is contained in a proper progression of dimension $s \leq C_1$ and size $L\leq C_2|A|$.
\end{thm}

We also need the following result of Pl\"unnecke--Ruzsa type \cite[Lemma~3.1]{KP20}.

\begin{lem} \label{lem:pr}
Let $G$ be an abelian group. If sets $A, B \subseteq G$ with $|A| = |B|$ are such that $C := A + B$ satisfies $|C| \leq K|A|$ for some $K > 0$, then $|C + C| \leq K^6 |C|$.
\end{lem}

Finally, we need the following technical lemma, saying that if $\cL \in\Mat_d(\QQ)$ has no non-trivial invariant subspace over $\QQ$ and $A$ is a finite subset of $\ZZ^d$ with $|A+\cL A|\leq K|A|$, then $A$ cannot be concentrated on an affine subspace.

\begin{lem} \label{lem:noplane}
Let $\cL\in\Mat_d(\QQ)$ with no non-trivial invariant subspace over $\QQ$ 
and let $A\subset \ZZ^d$ be such that $|A|=n$ and $|A+\cL A|\leq Kn$ for some $K > 0$. If $U$ is a vector subspace of $\QQ^d$ of dimension $k<d$, then every translate of $U$ contains at most $(Kn)^{1-2^{-k}}$ points of $A$.
\end{lem}

\begin{proof}
The fact that $\cL$ has no non-trivial invariant subspace implies that $\cL$ is invertible (over $\QQ$). 
We prove the lemma by induction on $k$. For $k=1$, let $U_1$ be a 1-dimensional subspace. Then,  
since $\cL$ is invertible, $\cL U_1$ is a line. Furthermore, the line $\cL U_1$ is not parallel to $U_1$, since $U_1$ is not an invariant subspace of $\cL$. Thus, for any translate $U_1+u$ of $U_1$, $|(U_1+u)\cap A|^2 = |((U_1+u)\cap A)+\cL((U_1+u)\cap A)|\leq Kn$, so $|(U_1+u)\cap A|\leq (Kn)^{1/2}$. This proves the base case of our induction.

For $1<k<d$, let $U_k$ be a subspace of dimension $k$. Then $\cL U_k \neq U_k$ since $\cL$ has no non-trivial invariant subspace, 
so $V=\cL U_k \cap U_k$ is a subspace of dimension strictly smaller than $k$. Let $U_k+u$ be a translate of $U_k$ with $|(U_k+u)\cap A|=m$. Suppose $r$ translates of $V$ are required to cover $(U_k+u)\cap A$. Note that for any collection of translates $V'$ of $V$, the affine subspaces $V'+\cL(U_k+u)$ are translates of $\cL U_k$ and are disjoint. Thus, $Kn\geq |((U_k+u)\cap A)+\cL((U_k+u)\cap A)|\geq mr$. On the other hand, each translate of $V$ intersects $A$ in at most $(Kn)^{1-2^{1-k}}$ points by the induction hypothesis. Thus, $m\leq r(Kn)^{1-2^{1-k}}$. Using $mr \leq Kn$, it follows that $m^2\leq (Kn)^{2-2^{1-k}}$, so $m\leq (Kn)^{1-2^{-k}}$, as desired.
\end{proof}

We now come to our application of Lemma~\ref{lem:discbm}.

\begin{lem} \label{lem:discbm2}
Let $\cL\in\Mat_d(\QQ)$ with no non-trivial invariant subspace over $\QQ$ 
and let $A\subset \ZZ^d$ be such that $|A+\cL A|\leq K|A|$ for some $K>0$. Then there are constants $D,\sigma>0$ depending only on $d$ and $K$ such that, for any $B_1\subseteq A$, $B_2\subseteq \cL A$, 
$$|B_1+B_2|\geq \paren{|B_1|^{1/d}+|B_2|^{1/d}}^d-D|A|^{1-\sigma}.$$
\end{lem}

\begin{proof}
Let $n=|A|$. By Lemma~\ref{lem:pr}, $|A+\cL A+A+\cL A|\leq K^6|A+\cL A|\leq K_1n$, where $K_1=K^7$. 
We also claim that 
$$|A+\cL A+\cL(A+\cL A)|\leq K_2 n,$$
where $K_2=K_1^2$. To see this, first note that $|A+\cL A+\cL A|\leq |A+\cL A+A+\cL A|\leq K_1n$ and $|\cL A+\cL A+\cL^2 A|=|A+A+\cL A|\leq |A+\cL A+A+\cL A|\leq K_1n$. By applying the sum version of Ruzsa's triangle inequality~\cite{R96}, which states that 
$$|A_1||A_2+A_3|\leq |A_1+A_2||A_1+A_3|$$
for any finite subsets $A_1,A_2,A_3$ of an abelian group, to the sets $A_1=\cL A,A_2=A+\cL A,A_3=\cL A+\cL^2 A$, we have
$$n|A+\cL A+\cL A+\cL^2 A|\leq |A+\cL A+\cL A||\cL A+\cL A+\cL^2 A|\leq K_1^2 n^2.$$
Thus, $|A+\cL A+\cL(A+\cL A)|\leq K_1^2n$, as claimed.

Since $|A+\cL A+A+\cL A|\le K_1n$, we can apply Theorem~\ref{thm:gr} to conclude that $A+\cL A$ is contained in a proper progression
$$P=\setcond{v_0+u_1v_1+\cdots+u_s v_s}{0\leq u_i<L_i \mbox { for } 1\leq i\leq s},$$
where $s\leq K_3$, $L_1\geq L_2\geq\cdots\geq L_s$ and $L_1L_2\cdots L_s\leq K_4n$ for some $K_3,K_4$ depending only on $K$. Note that $P$ cannot be contained in a hyperplane, since otherwise it would contradict Lemma~\ref{lem:noplane}.

Let $i_1=1$ and, for $j=2,\ldots,d$, set $i_j$ to be the smallest number such that $v_{i_j}$ does not lie in the span of $v_{i_1},\ldots,v_{i_{j-1}}$. Then $v_{i_1},\ldots,v_{i_d}$ forms a basis of $\RR^d$. By applying Lemma~\ref{lem:discbm} with this basis, we get that
\begin{align*}
    |B_1+B_2| &\geq \paren{|B_1|^{1/d}+|B_2|^{1/d}}^d-\sum_{I\subsetneq [d]}|p_I(B_1+B_2)|\\
    &\geq \paren{|B_1|^{1/d}+|B_2|^{1/d}}^d-2^d(|p_1(B_1+B_2)|+\cdots+|p_d(B_1+B_2)|),
\end{align*}
where $p_j=p_{[d]\setminus\{j\}}$, the projection along the basis element $v_{i_j}$. Hence, it suffices to show that there is some $\sigma>0$ such that $|p_j(A+\cL A)|=O(n^{1-\sigma})$ for all $j$. 

Note that $|p_j(A+\cL A)|\leq L_1\cdots L_s/L_{i_j}\leq K_4n/L_{i_j}$. Let $H$ be the span of $v_1,v_2,\ldots,v_{i_j-1}$, which is a proper subspace. 
Using the claim that $|A+\cL A+\cL(A+\cL A)|\leq K_2n$, we can apply Lemma~\ref{lem:noplane} with $A$ replaced by $A+\cL A$ to conclude that each translate of $H$ contains at most $(K_2n)^{1-2^{1-d}}$ points of $A+\cL A$. But $P$ is covered by $L_{i_j}L_{i_j+1}\cdots L_s$ translates of $H$. Hence,
$$L_{i_j}L_{i_j+1}\cdots L_s\geq n/(K_2n)^{1-2^{1-d}}=K_5n^{2^{1-d}},$$
where $K_5=K_2^{2^{1-d}-1}$. Since $L_{i_j}\geq L_{i_j+1}\geq \cdots\geq L_s$, we have
$$L_{i_j}\geq K_5^{1/s}n^{2^{1-d}/s}\geq K_6n^{2^{1-d}/K_3},$$
where $K_6=K_5^{1/K_3}$. Thus, 
$$|p_j(A+\cL A)|\leq K_4n/L_{i_j}\leq \frac{K_4}{K_6} n^{1-2^{1-d}/K_3}.$$
The result therefore follows by taking $\sigma=2^{1-d}/K_3$ and $D=2^d dK_4/K_6$.
\end{proof}

\section{Bounding $A+\cL A$} \label{sec:iden}

As promised, we will first prove our main result in the special case where one of the transformations is the identity. Like the  general case, we will do this by proving a bootstrapping lemma which allows us to successively obtain better and better bounds, approaching the optimal one. We start with a weaker version of this bootstrapping lemma.

Both here and in what follows, we will make extensive use of the fact that if $\cL$ is not singular, then $\cL\ZZ^d$ has index $k=|\det\cL|$ in $\ZZ^d$. Indeed, this can be seen by considering the Smith normal form  $\cL=SDT$, where $S,T\in \Mat_d(\ZZ)$ are invertible over $\ZZ$ 
and $D\in \Mat_d(\ZZ)$ is diagonal. Then the index satisfies
\[    [\ZZ^d:\cL\ZZ^d] =[S^{-1}\ZZ^d:DT\ZZ^d]
    =[\ZZ^d:D\ZZ^d]
    =|\det D|=|\det \cL|.
\]

\begin{lem} \label{lem:bootstrap1}
Let $\cL\in \Mat_d(\ZZ)$ have no non-trivial invariant subspace over $\QQ$
and $k=|\det \cL|$. Then there are constants $\sigma_1>0$ and $D>0$ depending only on $d$ and $k$ 
such that the following holds. Suppose that there are $0<\alpha < (1+k^{1/d})^d$ and $D_1>0$ such that
$$|A+\cL A|\geq ((1+k^{1/d})^d-\alpha)|A|-D_1|A|^{1-\sigma_1}$$
holds for all finite $A\subset \ZZ^d$. Let $I_1,\ldots,I_k$ be the cosets of $\cL\ZZ^d$ in $\ZZ^d$ and let $A_i=A\cap I_i$ for $i=1,\ldots,k$. If there is some $j$ for which $0<|A_j|\leq |A|/k$, then
$$|A+\cL A| \geq \paren{(1+k^{1/d})^d-\max\paren{\alpha-1,\frac{k-1}{k}\alpha}}|A|-(D+(k-1)D_1)|A|^{1-\sigma_1}$$
holds for all finite $A\subset \ZZ^d$.
\end{lem}

\begin{proof}
Assume that $|A+\cL A|\leq (1+k^{1/d})^d|A|$. Let $D$ and $\sigma_1=\sigma$ be the constants obtained from applying Lemma~\ref{lem:discbm2} with $K = (1+k^{1/d})^d$. Then, for each $i$, we have
$$|A_i+\cL A|\geq \paren{|A_i|^{1/d}+|A|^{1/d}}^d-D|A|^{1-\sigma_1}.$$

Since $\cL A\subset \cL \ZZ^d$, we have $(A+\cL A)\cap I_i=A_i+\cL A$. Hence, 
we can write $A+\cL A$ as the disjoint union
$$A+\cL A=(A_1+\cL A)\cup\cdots \cup (A_k+\cL A).$$
Suppose, without loss of generality, that $0<|A_1|\leq |A|/k$. We shall bound $|A_1+\cL A|$ by the estimate above and the rest by
\begin{align*}
|A_i+\cL A| &\geq |A_i+\cL A_i|\geq ((1+k^{1/d})^d-\alpha)|A_i|-D_1|A_i|^{1-\sigma_1}\\
&\geq ((1+k^{1/d})^d-\alpha)|A_i|-D_1|A|^{1-\sigma_1}
\end{align*}
for $i=2,\ldots,k$. Combining these estimates, we have
\begin{align*}
|A+\cL A| &\geq |A_1+\cL A|+|A_2+\cL A|+\dots +|A_k+\cL A|\\
&\geq \paren{|A_1|^{1/d}+|A|^{1/d}}^d+((1+k^{1/d})^d-\alpha)\sum_{i=2}^k |A_i|-(D+(k-1)D_1)|A|^{1-\sigma_1}\\
&= \paren{|A_1|^{1/d}+|A|^{1/d}}^d+((1+k^{1/d})^d-\alpha)(|A|-|A_1|)-(D+(k-1)D_1)|A|^{1-\sigma_1}.
\end{align*}
This last expression is concave in terms of $|A_1|$, which can be seen by expanding the binomial term and noting that each term in the binomial sum is concave. Hence, it is minimised when $|A_1|=0$ or $|A_1|=|A|/k$.

In the first case, where  the minimum is when $|A_1|=0$, we have
$$|A+\cL A|\geq ((1+k^{1/d})^d-(\alpha-1))|A|-(D+(k-1)D_1)|A|^{1-\sigma_1}.$$
In the second case, where  the minimum is when $|A_1|=|A|/k$, we have 
\begin{align*}
|A+\cL A| &\geq \paren{(1+k^{1/d})^d-\frac{k-1}{k}\alpha}|A|-(D+(k-1)D_1)|A|^{1-\sigma_1}.
\end{align*}
In either case, we have
$$|A+\cL A| \geq \paren{(1+k^{1/d})^d-\max\paren{\alpha-1,\frac{k-1}{k}\alpha}}|A|-(D+(k-1)D_1)|A|^{1-\sigma_1},$$ 
as required.
\end{proof}

This lemma shows that bootstrapping works if each of the $k$ cosets $A_i$ of $A$ are non-empty. To show that a similar result holds in general, we split each of the cosets $A_i$ into smaller cosets $A_{ij}$. There are then three cases: if $A$ is contained in some smaller sublattice, then we can rescale $A$, which will contradict a certain minimality assumption; if $A+\cL A$ contains cosets that are distinct from all the $A_{ij}+\cL A_{ij}$, then this  additional coset boosts the bound; and, finally, if any of the $A_i$ splits into $k$ non-empty cosets, we can again apply the lemma above. 
The following lemma will allow us to show that one of these three cases must hold.

\begin{lem} \label{lem:tri}
Let $\cL\in\Mat_d(\ZZ)$ be a linear transformation that is invertible over $\QQ$. Let $X$ be a subset of the finite abelian group $G=\ZZ^d/\cL^2\ZZ^d$ containing 0 and let $H$ be the subgroup $\cL\ZZ^d/\cL^2\ZZ^d$ of $G$. Notice that $\cL$ naturally induces a map $G\to G$. Then at least one of the following holds:
\begin{enumerate}
    \item $X+H$ does not generate $G$;
    \item $X+\cL X\supsetneq X$ (note that $X+\cL X\supseteq X$ always holds);
    \item $H\subseteq X$.
\end{enumerate}
\end{lem}

\begin{proof}
Suppose all 3 do not hold. Let $L=\setcond{v\in G}{\cL v\in X}$. Since $0\in X$, we have $H\subseteq L$. For any $v\in L$ and $a\in X$, we have $\cL v+\cL a\in X+\cL X=X$, so $v+a\in L$. 
Since $H+X$ generates $G$, for any $b\in G$, there are $h\in H$ and $a_1,\ldots,a_k\in X$ for some $k$ such that $b=h+a_1+\cdots+a_k$. If $h+a_1+\cdots+a_i\in L$ for some $i$, then $(h+a_1+\cdots+a_i)+a_{i+1}\in L$, so, by the fact that $h\in L$ and a simple induction, we have that $b=h+a_1+\cdots+a_k\in L$. Thus, $L=G$, which implies that $H\subseteq X$, a contradiction.
\end{proof}

We are now ready for our main bootstrapping lemma.

\begin{lem} \label{lem:bootstrap2}
Let $d$ and $k$ be positive integers. Then there are constants $\sigma_1>0$ and $D>0$ depending only on $d$ and $k$ such that the following holds. Suppose that there are $0<\alpha < (1+k^{1/d})^d$ and $D_1>0$ such that 
$$|A+\cL A|\geq ((1+k^{1/d})^d-\alpha)|A|-D_1|A|^{1-\sigma_1}$$
holds for all finite $A\subset \ZZ^d$ and all $\cL\in \Mat_d(\ZZ)$ with no non-trivial invariant subspace over $\QQ$ and $k=|\det \cL|$. Then
$$|A+\cL A|\geq \paren{(1+k^{1/d})^d-\max\paren{\alpha-\frac{1}{k^2},\frac{k^2-1}{k^2}\alpha}}|A|-(D+k^2D_1)|A|^{1-\sigma_1}$$
holds for all such $A$ and $\cL$.
\end{lem}

\begin{proof}
Take $\sigma_1,D$ as in Lemma \ref{lem:bootstrap1}. By translating $A$, we may assume that $0\in A$. We may also assume that $|A+\cL A|\leq (1+k^{1/d})^d|A|$, so that, by Lemma~\ref{lem:noplane}, $A$ does not lie on a hyperplane. Let $\ang{A}$ denote the $\ZZ$-span of $A$, which is a $d$-dimensional sublattice of $\ZZ^d$. Suppose the lemma does not hold and pick a counterexample $(A,\cL)$ such that $\ang{A}$ has minimum index in $\ZZ^d$. 

Let $v_1=0,v_2,\ldots,v_k$ be coset representatives of $\cL \ZZ^d$ over $\ZZ^d$. For $i,j=1,\ldots,k$, let $A_i=A\cap (v_i+\cL\ZZ^d)$ and $A_{ij}=A\cap (v_i+\cL v_j+\cL^2\ZZ^d)$. Then the $A_{ij}$ partition $A_i$ and the $A_i$ partition $A$. If there is some $i$ for which $0<|A_i|\leq |A|/k$, then we are done by Lemma~\ref{lem:bootstrap1}. Hence, we may assume that either $A_i=\emptyset$ or $|A_i|>|A|/k$ for every $i$. 

If there is some $i,j$ such that $A_i\neq\emptyset$ and $0<|A_{ij}|\leq |A_i|/k$, then let $A'=\cL^{-1}(A_i-v_i)=\cL^{-1}(A-v_i)\cap \ZZ^d\subseteq \ZZ^d$. For each $l =1,\dots,k$, let $A_l'=\cL^{-1}(A_{il}-v_i)=\cL^{-1}(A-v_i)\cap (v_l+\cL\ZZ^d)$. Thus, $A_l'=A'\cap (v_l+\cL\ZZ^d)$. Hence, applying Lemma~\ref{lem:bootstrap1} with $A$ and $A_j$ replaced by $A'$ and $A_j'$, we have 
\begin{align*}
    |A_i+\cL A_i| &= |A'+\cL A'|\\
    &\geq \paren{(1+k^{1/d})^d-\max\paren{\alpha-1,\frac{k-1}{k}\alpha}}|A_i|-(D+(k-1)D_1)|A_i|^{1-\sigma_1}.
\end{align*}
Using the fact that $|A+\cL A|\geq \sum_{l=1}^k |A_l+\cL A_l|$ and $|A_i|\geq |A|/k$, we have 
\begin{align*}
|A+\cL A| &\geq \sum_{l\neq i} |A_l+\cL A_l|+|A_i+\cL A_i|\\
&\geq \paren{(1+k^{1/d})^d-\alpha}\sum_{l\neq i}|A_l|-(k-1)D_1|A|^{1-\sigma_1}\\
&\quad +\paren{(1+k^{1/d})^d-\max\paren{\alpha-1,\frac{k-1}{k}\alpha}}|A_i|-(D+(k-1)D_1)|A|^{1-\sigma_1}\\
&= \paren{(1+k^{1/d})^d-\alpha}|A|+\min\paren{1,\frac{\alpha}{k}}|A_i|-(D+2(k-1)D_1)|A|^{1-\sigma_1}\\
&\geq \paren{(1+k^{1/d})^d-\alpha}|A|+\min\paren{\frac{1}{k},\frac{\alpha}{k^2}}|A|-(D+2(k-1)D_1)|A|^{1-\sigma_1}\\
&\geq \paren{(1+k^{1/d})^d-\max\paren{\alpha-\frac{1}{k},\frac{k^2-1}{k^2}\alpha}}|A|-(D+k^2D_1)|A|^{1-\sigma_1}.
\end{align*}
Hence, we may assume that, for all $i,j$, either $A_{ij}=\emptyset$ or $|A_{ij}|> |A_i|/k> |A|/k^2$. This assumption will be crucial in many of the estimates that follow.

Let $X$ be the image of $A$ in $G=\ZZ^d/\cL^2\ZZ^d$ and let $H=\cL\ZZ^d/\cL^2\ZZ^d\subseteq G$. Applying Lemma \ref{lem:tri} to $X$, we have the following 3 cases:\\

\noindent
\underline{Case 1: $X+H$ does not generate $G$}

\noindent
Let $E\subset \ZZ^d$ be the lattice that is the preimage of the subgroup of $G$ generated by $X+H$  with respect to the quotient map $q:\ZZ^d\to \ZZ^d/\cL^2\ZZ^d=G$. In other words, $E=\ang{A}+\cL\ZZ^d$. Since $A$ does not lie on a hyperplane, $E$ is $d$-dimensional and, since $X+H$ does not generate $G$, $E\neq \ZZ^d$. Consider a linear transformation $\cP\in\Mat_d(\ZZ)$ such that $\cP\ZZ^d=E$, so that $|\det \cP|>1$. Then
$$|A+\cL A|=|\cP\cP^{-1} A+\cP\cP^{-1}\cL\cP\cP^{-1} A|=|\cP^{-1} A+(\cP^{-1}\cL\cP)(\cP^{-1} A)|.$$
Since $\cL E\subset \cL\ZZ^d=q^{-1}(H)\subseteq E$, we have
$$\cP^{-1}\cL\cP \ZZ^d=\cP^{-1}\cL E\subset \cP^{-1}E=\ZZ^d,$$
so that $\cP^{-1}\cL\cP\in \Mat_d(\ZZ)$ and $|\det \cP^{-1}\cL\cP|=|\det \cP|^{-1}|\det\cL||\det\cP|=k$. Now replace $A$ by $\cP^{-1} A\subset \ZZ^d$ and $\cL$ by $\cP^{-1}\cL\cP$. But then the index of $\ang{\cP^{-1}A}$ is
$$[\ZZ^d:\ang{\cP^{-1}A}]=[\cP\ZZ^d:\ang{A}]=[\ZZ^d:\ang{A}]/[\ZZ^d:\cP\ZZ^d]=|\det \cP|^{-1}[\ZZ^d:\ang{A}].$$
Thus, $\ang{\cP^{-1}A}$ has strictly smaller index than $\ang{A}$, so the pair $(\cP^{-1}A,\cP^{-1}\cL\cP)$ contradicts the minimality of the pair $(A,\cL)$.\\

\noindent
\underline{Case 2: $X+\cL X\supsetneq X$}

\noindent
This case is saying that $A+\cL A$ intersects  strictly more cosets of $\cL^2\ZZ^d$ than $A$, so we can exploit the extra cosets to obtain a better lower bound. Let $I=\setcond{(i,j)\in [k]^2}{A_{ij}\neq\emptyset}$. Suppose $(i,j),(i',j')\in I$ are distinct pairs. We claim that $A_{ij}+\cL A_{ij}$ and $A_{i'j'}+\cL A_{i'j'}$ belong to different cosets of $\cL^2\ZZ^d$. Indeed, suppose they belong to the same coset. $A_{ij}+\cL A_{ij}$ belongs to the coset $v_i+\cL v_j+\cL v_i+\cL^2 \ZZ^d$, while $A_{i'j'}+\cL A_{i'j'}\subset v_{i'}+\cL v_{j'}+\cL v_{i'}+\cL^2 \ZZ^d$. So if they belong to the same coset, we must have $i=i'$ and $j=j'$. Now, since $A+\cL A$ intersects more than $|I|$ cosets, there are $(i_1,j_1),(i_2,j_2)\in I$ such that $A_{i_1j_1}+\cL A_{i_2j_2}$ belongs to a coset different from $A_{ij}+\cL A_{ij}$ for all $(i,j)\in I$. Since $A_{i_1j_1}$ is non-empty, $|A_{i_1j_1}|\geq |A|/k^2$, so we have
\begin{align*}
|A+\cL A| &\geq \sum_{(i,j)\in I}|A_{ij}+\cL A_{ij}|+|A_{i_1j_1}+\cL A_{i_2j_2}|\\
&\geq ((1+k^{1/d})^d-\alpha)|A|-k^2D_1|A|^{1-\sigma_1}+|A_{i_1j_1}|\\
&\geq ((1+k^{1/d})^d-\alpha)|A|-k^2D_1|A|^{1-\sigma_1}+\frac{1}{k^2}|A|\\
&= ((1+k^{1/d})^d-(\alpha-1/k^2))|A|-k^2D_1|A|^{1-\sigma_1}.
\end{align*}\\

\noindent
\underline{Case 3: $H\subseteq X$}

\noindent
In this case, $A_{1j}\neq\emptyset$ for $j=1,\ldots,k$. But, since the $A_{1j}$ partition $A_1$, there is then some $j$ for which $|A_{1j}|\leq |A_1|/k$, contradicting our assumption. This completes the proof of the lemma.
\end{proof}

It is now a simple matter to complete the proof of Theorem~\ref{thm:main} in the special case where $\cL_1$ is the identity.

\begin{thm} \label{thm:main1}
Let $\cL\in \Mat_d(\ZZ)$ be a linear transformation with no non-trivial invariant subspace over $\QQ$ and $k=|\det \cL|$. Then there are $D_2,\sigma_2>0$ depending only on $d$ and $k$
such that
$$|A+\cL A| \geq (1+k^{1/d})^d|A|-D_2|A|^{1-\sigma_2}$$
for all finite $A\subset \ZZ^d$.
\end{thm}

\begin{proof}
Let $\sigma_1,D>0$ be as in Lemma~\ref{lem:bootstrap2}. Using the trivial base case $|A+\cL A|\geq |A|$ and repeatedly applying Lemma~\ref{lem:bootstrap2}, we can find some $0<\epsilon<1$ and $D_2'>D$ such that 
$$|A+\cL A|\geq ((1+k^{1/d})^d-\epsilon)|A|-D_2'|A|^{1-\sigma_1}$$
holds for all finite $A\subset\ZZ^d$.

Applying Lemma~\ref{lem:bootstrap2} $m$ more times, we have
$$|A+\cL A| \geq \paren{(1+k^{1/d})^d-\paren{\frac{k^2-1}{k^2}}^m\epsilon}|A|-(k^2+1)^mD_2'|A|^{1-\sigma_1}.$$
Taking $m=\frac{\sigma_1 \log |A|}{2\log(k^2+1)}$ (and ignoring integer rounding issues), we have
$$(k^2+1)^mD_2'|A|^{1-\sigma_1}=D_2'|A|^{1-\sigma_1/2}$$
and
\begin{align*}
\paren{\frac{k^2-1}{k^2}}^m\epsilon |A| &= \epsilon |A|^{1+\frac{\sigma_1(\log (k^2-1)-\log k^2)}{2\log (k^2+1)}}.
\end{align*}
Now, taking $\sigma_2=\min\paren{\frac{\sigma_1}{2},\frac{\sigma_1(\log k^2-\log (k^2-1))}{2\log (k^2+1)}}$, we get
$$|A+\cL A|\geq (1+k^{1/d})^d|A|-D_2|A|^{1-\sigma_2},$$
where $D_2=\epsilon+D_2'$.
\end{proof}

\section{Bounding $\cL_1 A+\cL_2 A$} \label{sec:main}

In this section, we prove our main result, our lower  bound on $|\cL_1  A + \cL_2 A|$ when $\cL_1,\cL_2\in\Mat_d(\ZZ)$ are irreducible and coprime. Note that we may assume that both $\cL_1$ and $\cL_2$ are invertible over $\QQ$. Indeed, if $\cL_1$, say, is not invertible, then there is a line $L$ such that $\cL_1 L=0$, so $\cL_1,\cL_2$ would not be irreducible. 

We first note the following elementary fact about abelian groups.

\begin{lem} \label{lem:subgrp}
Let $G$ be an abelian group and $H_1, H_2$ be subgroups of finite index such that $H_1+H_2=G$. Then
$$[G:H_1\cap H_2]=[G:H_1][G:H_2].$$
\end{lem}

\begin{proof}
By the isomorphism theorems, we have that
$$H_1/(H_1\cap H_2)\equiv (H_1+H_2)/H_2,$$
$$G/H_1\equiv (G/(H_1\cap H_2))/(H_1/(H_1\cap H_2)).$$
Hence, $[H_1:H_1\cap H_2]=[G:H_2]$ and $[G:H_1\cap H_2]=[G:H_1][H_1:H_1\cap H_2]=[G:H_1][G:H_2]$.
\end{proof}

For the proof, we will need to introduce a number of additional linear transformations  associated with $\cL_1$ and $\cL_2$. Indeed, let $p=|\det\cL_1|$ and $q=|\det\cL_2|$. Since $\cL_1, \cL_2$ are coprime, we know that 
$\cL_1\ZZ^d+\cL_2\ZZ^d=\ZZ^d$. Thus, by Lemma~\ref{lem:subgrp} with $G = \ZZ^d$, $H_1 = \cL_1\ZZ^d$ and $H_2 = \cL_2\ZZ^d$, we have 
$$[\ZZ^d:\cL_1 \ZZ^d\cap \cL_2 \ZZ^d] = [\ZZ^d:\cL_1 \ZZ^d] [\ZZ^d:\cL_2 \ZZ^d] = pq.$$
Hence,
$$[\cL_1\ZZ^d:\cL_1 \ZZ^d\cap \cL_2 \ZZ^d]=q,\quad [\cL_2 \ZZ^d:\cL_1 \ZZ^d\cap \cL_2 \ZZ^d]=p$$
and so
$$[\ZZ^d:\ZZ^d\cap \cL_1^{-1}\cL_2 \ZZ^d]=q,\quad [\ZZ^d:\ZZ^d\cap \cL_2^{-1}\cL_1 \ZZ^d]=p.$$
We now let $\cP_1,\cP_2\in\Mat_d(\ZZ)$ be linear transformations such that $\cP_1\ZZ^d=\ZZ^d\cap \cL_2^{-1}\cL_1 \ZZ^d$ and $\cP_2\ZZ^d=\ZZ^d\cap \cL_1^{-1}\cL_2 \ZZ^d$, noting that $|\det\cP_1|=p$ and $|\det\cP_2|=q$.

As in the $A+\cL A$ case, we begin the proof proper with a weak bootstrapping lemma.

\begin{lem} \label{lem:bootstrap21}
Let $\cL_1,\cL_2\in\Mat_d(\ZZ)$ be irreducible, coprime linear transformations with $|\det\cL_1|=p$ and $|\det\cL_2|=q$.
Then there are constants $\sigma_1>0$ and $D>0$ depending only on $d$, $p$ and $q$ such that the following holds. Suppose that there are $0<\alpha<(p^{1/d}+q^{1/d})^d$ and $D_1>0$ such that 
$$|\cL_1 A+\cL_2 A|\geq ((p^{1/d}+q^{1/d})^d-\alpha)|A|-D_1|A|^{1-\sigma_1}$$
holds for all finite $A\subset \ZZ^d$. Let $I_1,\ldots,I_p$ be the cosets of $\cP_1\ZZ^d$ in $\ZZ^d$ and $I^1,\ldots,I^q$ the cosets of $\cP_2\ZZ^d$ and let $A_i=A\cap I_i$, $A^j=A\cap I^j$ and $A_i^j=A\cap I_i\cap I^j$. If either 
\begin{enumerate}
\item $A_i,A^j\neq\emptyset$ for all $1\leq i\leq p,1\leq j\leq q$ or
\item there are some $i,j$ such that $A_i,A^j\neq\emptyset$ and $|A_i^j|\leq c|A|$,
where $c=\frac{1}{2p(p^{1/d}+q^{1/d})^{2d}}$,
\end{enumerate}
then 
$$|\cL_1 A+\cL_2 A|\geq ((p^{1/d}+q^{1/d})^d-(1-c)\alpha)|A|-((p+q)D_1+D)|A|^{1-\sigma_1}$$
holds for all finite $A\subset \ZZ^d$.
\end{lem}

\begin{proof}
Since $\cL_1,\cL_2$ are irreducible, $\cL_1^{-1}\cL_2 \in \Mat_d(\QQ)$ has no non-trivial invariant subspace over $\QQ$. We may also assume that $|\cL_1 A+\cL_2 A|\leq (p^{1/d}+q^{1/d})^d|A|$, so that, by Lemma~\ref{lem:discbm2} with $\cL=\cL_1^{-1}\cL_2$, there are $\sigma_1,D>0$ such that, for any $B_1,B_2\subseteq A$,
$$|\cL_1 B_1+\cL_2 B_2|=|B_1+\cL B_2|\geq (|B_1|^{1/d}+|B_2|^{1/d})^d-D|A|^{1-\sigma_1}.$$

We claim that there is a choice of $i$ and $j$ such that $A_i,A^j\neq\emptyset$ and 
$$(|A_i|^{1/d}+|A^j|^{1/d})^d-((p^{1/d}+q^{1/d})^d-\alpha)|A_i^j|\geq \alpha c|A|.$$

Suppose first that $A_i,A^j\neq\emptyset$ for all $i,j$. Pick $i$ and $j$ such that $|A_i^j|$ is minimal. If $|A_i^j|\leq c|A|$, then we may pass to the second case. Otherwise, $|A_i^j|>c|A|$. Since, for any $i',j'$, we have $|A_{i'}^j|\geq |A_i^j|$ and $|A_i^{j'}|\geq |A_i^j|$, we see that $|A^j|\geq p|A_i^j|$ and $|A_i|\geq q|A_i^j|$. Hence,
$$(|A_i|^{1/d}+|A^j|^{1/d})^d-((p^{1/d}+q^{1/d})^d-\alpha)|A_i^j|\geq \alpha|A_i^j|\geq \alpha c|A|.$$

Suppose now that there are $i,j$ such that $A_i,A^j\neq\emptyset$ and $|A_i^j|\leq c|A|$. 
If there is some $i'$ such that $|A_{i'}^j|>(p^{1/d}+q^{1/d})^dc|A|$, then
\begin{align*}
(|A_i|^{1/d}+|A^j|^{1/d})^d-((p^{1/d}+q^{1/d})^d-\alpha)|A_i^j| &\geq |A_{i'}^j|-((p^{1/d}+q^{1/d})^d-\alpha)|A_i^j|\\
&\geq (p^{1/d}+q^{1/d})^dc|A|-((p^{1/d}+q^{1/d})^d-\alpha)c|A|\\
&= \alpha c|A|.
\end{align*}
Otherwise, we may assume that $|A_{i'}^j|\leq (p^{1/d}+q^{1/d})^dc|A|$ for all $i'$. Since $\sum_i |A_i|=|A|$, there is some $i'$ such that $|A_{i'}|\geq |A|/p$. Thus,
\begin{align*}
(|A_{i'}|^{1/d}+|A^j|^{1/d})^d-((p^{1/d}+q^{1/d})^d-\alpha)|A_{i'}^j| &\geq |A_{i'}|-((p^{1/d}+q^{1/d})^d-\alpha)|A_{i'}^j|\\
&\geq \frac{1}{p}|A|-((p^{1/d}+q^{1/d})^d-\alpha)(p^{1/d}+q^{1/d})^dc|A|\\
&\geq \frac{1}{2p}|A|>\alpha c|A|.
\end{align*}
This proves the claim.
From here on, without loss of generality, we will assume that $A_1,A^1\neq\emptyset$ and 
\begin{align} \label{eqn:eqn1}
    (|A_1|^{1/d}+|A^1|^{1/d})^d-((p^{1/d}+q^{1/d})^d-\alpha)|A_1^1|\geq \alpha c|A|.
\end{align}

We will now show that the sets $\cL_1 A+\cL_2 A_i$ belong to different cosets of $\cL_1\ZZ^d$ for $i=1,\ldots,p$ and so are disjoint. Note that $\cL_2\cP_1\ZZ^d\subseteq \cL_1\ZZ^d$, so the sets do indeed belong to cosets of $\cL_1\ZZ^d$. If, for some $i,i'$, the corresponding sets belong to the same coset, then $\cL_2 I_i-\cL_2 I_{i'} \subseteq \cL_1\ZZ^d$, so $I_i-I_{i'}\subseteq \cL_2^{-1}\cL_1\ZZ^d$. But this means that $I_i$ and $I_{i'}$ are the same coset of $\cP_1\ZZ^d$. Hence, $\cL_1 A+\cL_2 A$ can be partitioned into the sets $\cL_1 A+\cL_2 A_i$ for $i=1,\ldots,p$. Similarly, the sets $\cL_1 A^j+\cL_2 A_1$ belong to disjoint cosets of $\cL_2\ZZ^d$, so $\cL_1 A+\cL_2 A_1$ can be partitioned into the sets $\cL_1 A^j+\cL_2 A_1$ for $j = 1, 2, \dots, q$. 

Note now that, by our choice of $\sigma_1$ and $D$, we have
$$|\cL_1 A^1+\cL_2 A_1|\geq (|A^1|^{1/d}+|A_1|^{1/d})^d-D|A|^{1-\sigma_1}.$$
Thus, using our earlier claim, we have
\begin{align*}
|\cL_1 A+\cL_2 A| &= \sum_{i=2}^p |\cL_1 A+\cL_2 A_i|+\sum_{j=2}^q |\cL_1 A^j+\cL_2 A_1|+|\cL_1 A^1+\cL_2 A_1|\\
&\geq \sum_{i=2}^p |\cL_1 A_i+\cL_2 A_i|+\sum_{j=2}^q |\cL_1 A_1^j+\cL_2 A_1^j|+|\cL_1 A^1+\cL_2 A_1|\\
&\geq ((p^{1/d}+q^{1/d})^d-\alpha)(|A|-|A_1^1|)-(p+q)D_1|A|^{1-\sigma_1}\\
&\quad +(|A^1|^{1/d}+|A_1|^{1/d})^d-D|A|^{1-\sigma_1}\\
&\geq ((p^{1/d}+q^{1/d})^d-\alpha)|A|+\alpha c|A|-((p+q)D_1+D)|A|^{1-\sigma_1}\\
&= ((p^{1/d}+q^{1/d})^d-(1-c)\alpha)|A|-((p+q)D_1+D)|A|^{1-\sigma_1},
\end{align*}
as required, where we used (\ref{eqn:eqn1}) in going from the third to the fourth line.
\end{proof}

We now introduce some further notation. Indeed, let 
$\cP\in\Mat_d(\ZZ)$ be a linear transformation such that
$$\cP\ZZ^d=\ZZ^d\cap \cL_1^{-1}\cL_2 \ZZ^d\cap \cL_2^{-1}\cL_1 \ZZ^d=\cP_1\ZZ^d\cap \cP_2\ZZ^d.$$
Then, by Lemma~\ref{lem:subgrp},
\begin{align*}
    |\det\cP| &= [\ZZ^d:\cP\ZZ^d]
    = [\ZZ^d:\cP_1\ZZ^d+\cP_2\ZZ^d][\cP_1\ZZ^d+\cP_2\ZZ^d:\cP_1\ZZ^d\cap \cP_2\ZZ^d]\\
    &= [\ZZ^d:\cP_1\ZZ^d+\cP_2\ZZ^d][\cP_1\ZZ^d+\cP_2\ZZ^d:\cP_1\ZZ^d][\cP_1\ZZ^d+\cP_2\ZZ^d:\cP_2\ZZ^d]\\
    &\leq [\ZZ^d:\cP_1\ZZ^d][\ZZ^d:\cP_2\ZZ^d]
    = |\det \cP_1||\det \cP_2|
    = pq.
\end{align*}
Moreover, let $\cQ\in\Mat_d(\ZZ)$ be such that 
$$\cQ\ZZ^d=\cL_1\ZZ^d\cap \cL_2\ZZ^d,$$
so that, as above, $|\det\cQ|=|\det \cL_1||\det \cL_2|=pq$. 

Note that $\cL_1\cP\ZZ^d,\cL_2\cP\ZZ^d\subseteq \cL_1\ZZ^d\cap \cL_2\ZZ^d = \cQ\ZZ^d$, so $\cQ^{-1}\cL_1\cP\ZZ^d,\cQ^{-1}\cL_2\cP\ZZ^d\subseteq\ZZ^d$, implying that $\cQ^{-1}\cL_1\cP,\cQ^{-1}\cL_2\cP\in\Mat_d(\ZZ)$.  Therefore, since $\cL_1,\cL_2$ are coprime, 
$$|\det \cQ^{-1}\cP|\geq 1.$$
But $|\det \cQ|=pq$ and $|\det \cP|\leq pq$, so we must have $|\det\cP|=pq$. 

Finally, we let $L_1$ be the lattice $\cP\ZZ^d\cap \cL_2^{-1}\cL_1\cP\ZZ^d$ and $L_2=\cP\ZZ^d\cap \cL_1^{-1}\cL_2\cP\ZZ^d$. 
The next lemma will be important in the proof of Lemma~\ref{lem:tri2} below, which, like Lemma~\ref{lem:tri} in the last section, says that any set $A$ falls into one of three categories, each helpful for our bootstrap.

\begin{lem} \label{lem:isom}
The linear maps $\cL_1,\cL_2$ induce homomorphisms 
$$\phi_1,\phi_2:\ZZ^d/L_1\to \ZZ^d/\cL_1\cP\ZZ^d$$
of finite abelian groups. 
Furthermore, $\phi_1+\phi_2$ is an isomorphism.
\end{lem}

\begin{proof}
If $x\in L_1$, then $\cL_1 x\in \cL_1\cP\ZZ^d$ and $\cL_2 x\in \cL_1\cP\ZZ^d$, so $\phi_1$ and $\phi_2$ are well-defined group homomorphisms. Now let $\phi':\ZZ^d \to \ZZ^d/\cL_1\cP\ZZ^d$ be the map induced by $\cL_1+\cL_2$. 

We first show that $\ker \phi'=L_1$. We have already seen above that $\ker \phi'\supseteq L_1$. For the converse, suppose that $x\in\ker\phi'$, so that $\cL_1 x+\cL_2 x=\cL_1\cP y$ for some $y\in\ZZ^d$. This implies that
$$x=\cP y-\cL_1^{-1}\cL_2 x,$$
$$x=\cL_2^{-1}\cL_1\cP y-\cL_2^{-1}\cL_1 x.$$
Since $\cP y\in \cL_1^{-1}\cL_2 \ZZ^d$, the first equation implies that $x\in \ZZ^d\cap \cL_1^{-1}\cL_2\ZZ^d$. From the second equation, we have $x\in \ZZ^d\cap \cL_2^{-1}\cL_1\ZZ^d$, so that $x\in \ZZ^d\cap \cL_1^{-1}\cL_2\ZZ^d\cap \cL_2^{-1}\cL_1\ZZ^d=\cP\ZZ^d$. It then follows from applying the second equation again that $x\in \cL_2^{-1}\cL_1\cP\ZZ^d$, so that $x\in \cP\ZZ^d\cap \cL_2^{-1}\cL_1\cP\ZZ^d=L_1$, proving that $\ker \phi' = L_1$.

Hence, the induced map $\phi:\ZZ^d/L_1 \to \ZZ^d/\cL_1\cP\ZZ^d$ is injective. To show that it is in fact an isomorphism, we shall show that $|\ZZ^d/L_1|=|\ZZ^d/\cL_1\cP\ZZ^d|$. From injectivity, we have $|\ZZ^d/L_1|\leq |\ZZ^d/\cL_1\cP\ZZ^d|$, so it suffices to show that $|\ZZ^d/L_1|\geq |\ZZ^d/\cL_1\cP\ZZ^d|$. 

Let $\cR\in \Mat_d(\ZZ)$ be such that $\cR \ZZ^d=\cL_1\cP \ZZ^d+\cL_2\cP \ZZ^d$. Then $\cR^{-1}\cL_1 \cP\ZZ^d\subseteq \ZZ^d$ and $\cR^{-1}\cL_2 \cP\ZZ^d\subseteq \ZZ^d$, so $\cR^{-1}\cL_1 \cP$ and $\cR^{-1}\cL_2 \cP$ are integer matrices. By coprimality, we have $|\det \cR|\leq |\det \cP|$. By Lemma~\ref{lem:subgrp}, we have 
$$[\cL_1\cP \ZZ^d+\cL_2\cP \ZZ^d:\cL_1\cP \ZZ^d\cap \cL_2\cP \ZZ^d]=[\cL_1\cP \ZZ^d+\cL_2\cP \ZZ^d:\cL_1\cP \ZZ^d][\cL_1\cP \ZZ^d+\cL_2\cP \ZZ^d:\cL_2\cP \ZZ^d].$$
In other words, $[\ZZ^d:\cL_2 L_1]|\det \cR|=[\ZZ^d:\cL_1\cP \ZZ^d][\ZZ^d:\cL_2\cP \ZZ^d]$. Since $|\det \cR|\leq |\det \cP|$, it follows from $[\ZZ^d:\cL_2 L_1]=|\det \cL_2|[\ZZ^d:L_1]$ and $[\ZZ^d:\cL_2 \cP\ZZ^d]=|\det \cL_2|[\ZZ^d:\cP\ZZ^d]$ that $[\ZZ^d:L_1]\geq [\ZZ^d:\cL_1\cP \ZZ^d]$, as required.
\end{proof}

Since $\cL_1\cP\ZZ^d$ has index $|\det \cL_1\cP|=p^2q$ and $\phi_1+\phi_2$ is an isomorphism, the lemma implies that $L_1$ also has index $p^2q$. Similarly, $L_2$ has index $pq^2$.

\begin{lem} \label{lem:tri2}
Let $X$ be a subset of $G=\ZZ^d/L_1$ containing $0$ and define $\phi_1,\phi_2$ as in the previous lemma. Then at least one of the following holds:
\begin{enumerate}
\item $X$ does not generate $G$;
\item $|\phi_1(X)+\phi_2(X)|>|X|$;
\item $\cP\ZZ^d/L_1\subseteq X$.
\end{enumerate}
\end{lem}

\begin{proof}
Suppose all 3 do not hold. Let $\phi=\phi_1+\phi_2$, which is an isomorphism by Lemma \ref{lem:isom}. Note that $\phi(X)\subseteq \phi_1(X)+\phi_2(X)$, so $|\phi_1(X)+\phi_2(X)|\geq |X|$ always holds. By  assumption, we must have $\phi_1(X)+\phi_2(X)=\phi(X)$. Hence, for any $x,y\in X$, we have $\phi^{-1}\phi_1(x)+\phi^{-1}\phi_2(y)\in X$. In particular, since $0\in X$, we have $\phi^{-1}\phi_1(x),\phi^{-1}\phi_2(x)\in X$.

We claim that $\phi^{-1}\phi_2(G)=\cP_2\ZZ^d/L_1$ and $\phi^{-1}\phi_1(\cP_2\ZZ^d/L_1)=\cP\ZZ^d/L_1$. For the first claim, note that $\phi_2(G)=\cL_2\ZZ^d/\cL_1\cP\ZZ^d$, so it suffices to show that $\phi(\cP_2\ZZ^d/L_1)=\cL_2\ZZ^d/\cL_1\cP\ZZ^d$. Note that, for any $x\in \cP_2\ZZ^d=\ZZ^d\cap \cL_1^{-1}\cL_2\ZZ^d$, we have $\cL_1 x,\cL_2 x\in \cL_2\ZZ^d$, 
so that $\phi(\cP_2\ZZ^d/L_1)\subseteq \cL_2\ZZ^d/\cL_1\cP\ZZ^d$. Since $\cP_2\ZZ^d$ and $\cL_2\ZZ^d$ have index $q$ and $L_1$ and $\cL_1\cP\ZZ^d$ have index $p^2q$, we have $|\cP_2\ZZ^d/L_1|=|\cL_2\ZZ^d/\cL_1\cP\ZZ^d|=p^2$. Since $\phi$ is an isomorphism, we must then have $\phi(\cP_2\ZZ^d/L_1)=\cL_2\ZZ^d/\cL_1\cP\ZZ^d$.

For the second claim,  note that  $\phi_1(\cP_2\ZZ^d/L_1)=\cL_1\cP_2\ZZ^d/\cL_1\cP\ZZ^d=(\cL_1\ZZ^d\cap \cL_2\ZZ^d)/\cL_1\cP\ZZ^d$, so it suffices to show that $\phi(\cP\ZZ^d/L_1)=(\cL_1\ZZ^d\cap \cL_2\ZZ^d)/\cL_1\cP\ZZ^d$. If $x\in \cP\ZZ^d$, then $\cL_1 x,\cL_2 x\in \cL_1\ZZ^d\cap \cL_2\ZZ^d$, so we have the inclusion $\phi(\cP\ZZ^d/L_1)\subseteq (\cL_1\ZZ^d\cap \cL_2\ZZ^d)/\cL_1\cP\ZZ^d$. By again counting sizes, we have $|\cP\ZZ^d/L_1|=|(\cL_1\ZZ^d\cap \cL_2\ZZ^d)/\cL_1\cP\ZZ^d|=p$, so  $\phi(\cP\ZZ^d/L_1)=(\cL_1\ZZ^d\cap \cL_2\ZZ^d)/\cL_1\cP\ZZ^d$, proving our claim.

Let $X'=X\cap \cP_2\ZZ^d/L_1$. Since $\phi^{-1}\phi_2(X)\subseteq X'$ and $X$ generates $G$, we have that $X'$ generates $\cP_2\ZZ^d/L_1$. Moreover, $\phi^{-1}\phi_1(X')\subseteq X$ and generates $\cP\ZZ^d/L_1$. Note that, for any $x\in \cP\ZZ^d/L_1$, $\phi_1(x)=0$, so $\phi^{-1}\phi_2(x)=x$. This implies that, for any $x\in X\cap \cP\ZZ^d/L_1$ and $y\in X'$, we have $\phi^{-1}\phi_1(y)+x=\phi^{-1}\phi_1(y)+\phi^{-1}\phi_2(x)\in X\cap \cP\ZZ^d/L_1$, so $X\cap \cP\ZZ^d/L_1$ is closed under adding elements of $\phi^{-1}\phi_1(X')$. But $\phi^{-1}\phi_1(X')$ generates $\cP\ZZ^d/L_1$ and $0\in X\cap \cP\ZZ^d/L_1$.
It follows that $X\cap \cP\ZZ^d/L_1=\cP\ZZ^d/L_1$, contradicting our third assumption.
\end{proof}

We now come to our main bootstrapping lemma.

\begin{lem} \label{lem:bootstrap22}
Let $d$, $p$ and $q$ be positive integers. Then there are constants $\sigma_1>0$ and $D>0$ depending only on $d$, $p$ and $q$ such that the following holds. Suppose that there are $0<\alpha<(p^{1/d}+q^{1/d})^d$ and $D_1>0$ such that 
$$|\cL_1 A+\cL_2 A|\geq ((p^{1/d}+q^{1/d})^d-\alpha)|A|-D_1|A|^{1-\sigma_1}$$
holds for all finite $A\subset \ZZ^d$ and all irreducible, coprime linear transformations $\cL_1,\cL_2\in\Mat_d(\ZZ)$ with $|\det\cL_1|=p$ and $|\det\cL_2|=q$. 
Then 
$$|\cL_1 A+\cL_2 A|\geq ((p^{1/d}+q^{1/d})^d-(1-c^2)\alpha)|A|-(4p^2q^2D_1+D)|A|^{1-\sigma_1}$$
holds for all such $A\subset \ZZ^d$ and $\cL_1,\cL_2$, where $c=\frac{1}{2\max(p,q)(p^{1/d}+q^{1/d})^{2d}}$.
\end{lem}

\begin{proof}
Take $\sigma_1,D$ as in Lemma \ref{lem:bootstrap21}. By translating $A$, we may assume that $0\in A$. We may also assume that $|\cL_1 A+\cL_2 A|\leq (p^{1/d}+q^{1/d})^d|A|$, so that, by Lemma~\ref{lem:noplane}, $A$ cannot lie on a hyperplane. Suppose now  that $A$ is a counterexample to the lemma with $[\ZZ^d:\ang{A}]$  minimal. Let $A'$ be the image of $A$ in $\ZZ^d/L_1$. By Lemma \ref{lem:tri2}, one of the following possibilities holds:
\begin{enumerate}
\item $A'$ does not generate $\ZZ^d/L_1$;
\item $|\phi_1(A')+\phi_2(A')|>|A'|$;
\item $\cP\ZZ^d/L_1\subseteq A'$.
\end{enumerate}
We consider each case separately.\\

\noindent
\underline{Case 1: (1) holds, but not (2)}

\noindent
The fact that (1) holds means that $\ang{A}+L_1\neq\ZZ^d$, so it must be a strictly smaller sublattice of $\ZZ^d$ of some index $k>1$. Let $\cQ\in\Mat_d(\ZZ)$ be such that $\cQ\ZZ^d=\ang{A}+L_1$, so that $|\det\cQ|=k$. Since (2) does not hold, we have $\phi_1(A')+\phi_2(A')=\phi(A')$, so $\ang{\phi_1(A')+\phi_2(A')}=\phi(\ang{A'})$. Since $\phi$ is an isomorphism and $\ang{A'}$ is a subgroup of $\ZZ/L_1$ of index $k$, $\phi(\ang{A'})$ is a subgroup of $\ZZ^d/\cL_1\cP\ZZ^d$ of index $k$. Thus, $\ang{\cL_1 A+\cL_2 A}+\cL_1\cP\ZZ^d$ is a sublattice of $\ZZ^d$ of index $k$.

Let $\cQ'\in\Mat_d(\ZZ)$ be such that $\cQ'\ZZ^d=\ang{\cL_1 A+\cL_2 A}+\cL_1\cP\ZZ^d$, so that $|\det\cQ'|=k$. Notice that
$$|\cL_1 A+\cL_2 A|=|\cQ'^{-1}\cL_1\cQ(\cQ^{-1}A)+\cQ'^{-1}\cL_2\cQ(\cQ^{-1}A)|,$$
so we may replace the triple $(\cL_1,\cL_2,A)$ with $(\cQ'^{-1}\cL_1\cQ,\cQ'^{-1}\cL_2\cQ,\cQ^{-1}A)$. It is easy to see that $\cQ'^{-1}\cL_1\cQ,\cQ'^{-1}\cL_2\cQ$ are still irreducible and coprime and $\cQ^{-1}A\subset\ZZ^d$. However, this contradicts the minimality of $[\ZZ^d:\ang{A}]$, since  $[\ZZ^d:\ang{\cQ^{-1}A}]=[\ZZ^d:\ang{A}]/k$. 
\\

\noindent
\underline{Case 2: (2) holds}

\noindent
Let $I_1,\ldots,I_{pq}$ be the cosets of $\cP\ZZ^d$ with $0\in I_1$ and let $A_i=A\cap I_i$ for $i=1,\ldots,pq$. Note that the cosets $I_i$ are the intersections of the cosets of $\cP_1\ZZ^d$ and the cosets of $\cP_2\ZZ^d$. If $0<|A_i|\leq c|A|$ for some $i$, then 
condition 2 of Lemma~\ref{lem:bootstrap21} implies that
$$|\cL_1 A+\cL_2 A|\geq ((p^{1/d}+q^{1/d})^d-(1-c)\alpha)|A|-((p+q)D_1+D)|A|^{1-\sigma_1}.$$
We may therefore assume that $|A_i|> c|A|$ whenever $A_i\neq\emptyset$. Let $I_{i,k}$ be the cosets of $\cP\ZZ^d\cap \cL_2^{-1}\cL_1\cP\ZZ^d$ in $I_i$ for $k=1,\ldots,p$, where $0\in I_{1,1}$, and let $A_{i,k}=A\cap I_{i,k}=A_i\cap I_{i,k}$.

Suppose that $|A_{i,k}|>c^2|A|$ whenever $A_{i,k}\neq\emptyset$. By (2), $|\phi_1(A')+\phi_2(A')|>|A'|=|\phi(A')|$. Hence, since $\phi(A')\subseteq \phi_1(A')+\phi_2(A')$, there are  $a_1,a_2\in A$ such that 
$$\cL_1 a_1+\cL_2 a_2\not\in (\cL_1+\cL_2)a+\cL_1\cP\ZZ^d$$
for all $a\in A$. Take $i_1,k_1,i_2,k_2$ such that $a_1\in A_{i_1,k_1},a_2\in A_{i_2,k_2}$, so they are both non-empty. Then 
$$\cL_1 A_{i_1,k_1}+\cL_2 A_{i_2,k_2}\subset \cL_1 a_1+\cL_2 a_2+\cL_1\cP\ZZ^d,$$
which is disjoint from any of the $\cL_1 A_{i,k}+\cL_2 A_{i,k}$. Therefore,
\begin{align*}
|\cL_1 A+\cL_2 A| &\geq \sum_{i=1}^{pq}\sum_{k=1}^p |\cL_1 A_{i,k}+\cL_2 A_{i,k}|+|\cL_1 A_{i_1,k_1}+\cL_2 A_{i_2,k_2}|\\
&\geq ((p^{1/d}+q^{1/d})^d-\alpha)|A|-p^2qD_1|A|^{1-\sigma_1}+|A_{i_1,k_1}|\\
&\geq ((p^{1/d}+q^{1/d})^d-(1-c^2)\alpha)|A|-p^2qD_1|A|^{1-\sigma_1}.
\end{align*}

Otherwise, by translating if necessary, we may assume that $|A_{1,1}|\leq c^2|A| \leq c|A_1|$ and $0\in A_{1,1}$. 
Let $\cQ\in\Mat_d(\ZZ)$ be such that $\cQ\ZZ^d=\cL_1\ZZ^d\cap \cL_2\ZZ^d$, so that $|\det\cQ|=pq$. Set $\cM_i=\cQ^{-1}\cL_i\cP$. Since $\cL_i\cP\ZZ^d \subseteq \cL_1\ZZ^d\cap \cL_2\ZZ^d=\cQ\ZZ^d$, we have $\cM_i\ZZ^d\subseteq \ZZ^d$ for $i=1,2$, so $\cM_i\in\Mat_d(\ZZ)$. 
Moreover, $\cM_1,\cM_2$ are irreducible and coprime, with determinants of absolute value $p$ and $q$, respectively. 

If we let $B=\cP^{-1}A_1\subset\ZZ^d$, our aim now is to apply Lemma~\ref{lem:bootstrap21} to the sum $\cM_1 B+\cM_2 B$. 
Indeed, if we replace $\cP_1,\cP_2$ in that lemma by $\cP_1',\cP_2'$ chosen so that $\cP_1'\ZZ^d=\ZZ^d\cap \cM_2^{-1}\cM_1\ZZ^d$ and $\cP_2'\ZZ^d=\ZZ^d\cap \cM_1^{-1}\cM_2\ZZ^d$, the set $A_1$ by $B_1 = \cP^{-1}A_{1,1}$, which, since $A_{1,1}=A_1\cap \cL_2^{-1}\cL_1\cP\ZZ^d$, satisfies $B_1=B\cap \cP^{-1}\cL_2^{-1}\cL_1\cP\ZZ^d=B\cap \cM_2^{-1}\cM_1\ZZ^d$, and $A_1^1$ by an appropriate non-empty subset $B_1^1\subseteq B_1$ (which is possible since $B_1$ contains 0), then we have $0<|B_1^1|\leq |B_1|\leq c|B|$, so condition 2 of Lemma~\ref{lem:bootstrap21} holds. 
Hence, by that lemma,
$$|\cM_1 B+\cM_2 B|\geq ((p^{1/d}+q^{1/d})^d-(1-c)\alpha)|B|-((p+q)D_1+D)|B|^{1-\sigma_1}.$$
This implies that 
\begin{align*}
|\cL_1 A_1+\cL_2 A_1| &= |\cQ^{-1}\cL_1\cP(\cP^{-1}A_1)+\cQ^{-1}\cL_2\cP (\cP^{-1}A_1)|\\
&= |\cM_1 B+\cM_2 B|\\
&\geq ((p^{1/d}+q^{1/d})^d-(1-c)\alpha)|B|-((p+q)D_1+D)|B|^{1-\sigma_1}\\
&\geq ((p^{1/d}+q^{1/d})^d-(1-c)\alpha)|A_1|-((p+q)D_1+D)|A|^{1-\sigma_1}.
\end{align*}
Thus, since $|A_1|\geq c|A|$,
\begin{align*}
|\cL_1 A+\cL_2 A| &\geq \sum_{i=2}^{pq}|\cL_1 A_i+\cL_2 A_i|+|\cL_1 A_1+\cL_2 A_1|\\
&\geq ((p^{1/d}+q^{1/d})^d-\alpha)(|A|-|A_1|)-pqD_1|A|^{1-\sigma_1}\\
&\quad +((p^{1/d}+q^{1/d})^d-(1-c)\alpha)|A_1|-((p+q)D_1+D)|A|^{1-\sigma_1}\\
&\geq ((p^{1/d}+q^{1/d})^d-\alpha)|A|+c\alpha|A_1|-((p+q+pq)D_1+D)|A|^{1-\sigma_1}\\
&\geq ((p^{1/d}+q^{1/d})^d-\alpha)|A|+c^2\alpha|A|-((p+q+pq)D_1+D)|A|^{1-\sigma_1}\\
&\geq ((p^{1/d}+q^{1/d})^d-(1-c^2)\alpha)|A|-(4p^2q^2D_1+D)|A|^{1-\sigma_1}.
\end{align*}

\vspace{2mm}
\noindent
\underline{Case 3: (3) holds}

\noindent
Let $A''$ be the image of $A$ in $\ZZ^d/L_2$. If we apply Lemma \ref{lem:tri2} to $A''$, but with the roles of $\cL_1,\cL_2$ swapped, we arrive at three similar cases. If either of the first two occurs, then we are again done as above. Otherwise, the third case holds, i.e., $\cP\ZZ^d/L_2\subseteq A''$. Define $A_1,\cM_1,\cM_2,B$ as in Case 2 and partition $B$ into $B_1\cup\cdots\cup B_p$, where the $B_i$ belong to different cosets of $\ZZ^d\cap \cM_2^{-1}\cM_1\ZZ^d$, and into $B^1\cup\cdots\cup B^q$, where the $B^j$ belong to different cosets of $\ZZ^d\cap \cM_1^{-1}\cM_2\ZZ^d$. 

Since $\cP\ZZ^d/L_1\subseteq A'$, we have $\cP\ZZ^d\subseteq A+L_1$ and so $\cP\ZZ^d\subseteq A_1+L_1$, since $A_1=A\cap \cP\ZZ^d$. Thus, $\ZZ^d=\cP^{-1}A_1+\cP^{-1}L_1$, which means that $B=\cP^{-1}A_1$
intersects every coset of $\cP^{-1}L_1=\ZZ^d\cap \cP^{-1}\cL_2^{-1}\cL_1\cP\ZZ^d=\ZZ^d\cap \cM_2^{-1}\cM_1\ZZ^d$, so all the $B_i$ are non-empty. Similarly, all of the $B^j$ are non-empty. Thus,  condition 1 of Lemma \ref{lem:bootstrap21} holds, so that
$$|\cM_1 B+\cM_2 B|\geq ((p^{1/d}+q^{1/d})^d-(1-c)\alpha)|B|-((p+q)D_1+D)|B|^{1-\sigma_1}.$$
The same calculation as  in Case 2 then shows that
\begin{align*}
|\cL_1 A+\cL_2 A| 
&\geq ((p^{1/d}+q^{1/d})^d-(1-c^2)\alpha)|A|-(4p^2q^2D_1+D)|A|^{1-\sigma_1},  
\end{align*}
as required. 
\end{proof}

\begin{thm} \label{thm:main2}
Let $\cL_1,\cL_2\in\Mat_d(\ZZ)$ be irreducible, coprime linear transformations with $|\det\cL_1|=p$ and $|\det\cL_2|=q$.
Then there are constants $\sigma_2>0$ and $D_2>0$ depending only on $d$, $p$ and $q$ such that 
$$|\cL_1 A+\cL_2 A|\geq (p^{1/d}+q^{1/d})^d|A|-D_2|A|^{1-\sigma_2}$$
for all finite $A\subset \ZZ^d$.
\end{thm}

\begin{proof}
This follows from Lemma~\ref{lem:bootstrap22} just as Theorem~\ref{thm:main1} follows from Lemma~\ref{lem:bootstrap2}.
\end{proof}

\section{The size of $A+\lambda \cdot A$ for algebraic $\lambda$} \label{sec:alg}

In this section, we prove Theorem~\ref{thm:alg}, our lower bound on $|A+\lambda \cdot A|$ for algebraic $\lambda \in\RR$. 
Though our estimate applies for all finite $A\subset \RR$,   
the following simple lemma of Krachun and Petrov \cite[Lemma 2.1]{KP20} allows us to restrict attention to sets $A\subset \QQ[\lambda]$.

\begin{lem} \label{lem:alg}
Suppose that $\lambda\in\CC$ and $A$ is a finite set of complex numbers. Then there exists a finite set $B\subset\QQ[\lambda]$ such that $|B| = |A|$ and $|B + \lambda \cdot B| \leq |A +\lambda \cdot A|$.
\end{lem}

Suppose now that $\lambda$ has minimal polynomial $p(x)=x^d+a_{d-1}x^{d-1}+\cdots +a_0\in\QQ[x]$. If  we view $\QQ[\lambda]$ as a $d$-dimensional $\QQ$-vector space with basis $1,\lambda,\lambda^2,\ldots,\lambda^{d-1}$, then multiplication by $\lambda$ is given by the linear transformation
$$\cL=\begin{pmatrix}
0 & 0 & 0 & \cdots & 0 & -a_0\\
1 & 0 & 0 & \cdots & 0 & -a_1\\
0 & 1 & 0 & \cdots & 0 & -a_2\\
\vdots & \vdots & \vdots & \ddots & \vdots & \vdots\\
0 & 0 & 0 & \cdots & 1 & -a_{d-1}
\end{pmatrix}\in \GL_d(\QQ).$$
Thus, the problem reduces to that of bounding $|A+\cL A|$ for $A\subset \QQ^d$. Let $b$ be the smallest positive integer such that $ba_i\in\ZZ$ for all $i=0,1,\ldots,d-1$. 
Then, if we let 
$$\cL_1=\begin{pmatrix}
1 & 0 & \cdots & 0 & 0\\
0 & 1 & \cdots & 0 & 0\\
\vdots & \vdots & \ddots & \vdots & \vdots\\
0 & 0 & \cdots & 1 & 0\\
0 & 0 & \cdots & 0 & b
\end{pmatrix},\,\quad 
\cL_2=\begin{pmatrix}
0 & 0 & 0 & \cdots & 0 & -ba_0\\
1 & 0 & 0 & \cdots & 0 & -ba_1\\
0 & 1 & 0 & \cdots & 0 & -ba_2\\
\vdots & \vdots & \vdots & \ddots & \vdots & \vdots\\
0 & 0 & 0 & \cdots & 1 & -ba_{d-1}
\end{pmatrix}\in \Mat_d(\ZZ),$$
we see that $|A+\cL A|=|\cL_1(\cL_1^{-1} A)+\cL_2(\cL_1^{-1} A)|$. Setting $B=\cL_1^{-1} A$, the problem becomes that of bounding $|\cL_1 B+\cL_2 B|$ for $B\subset \QQ^d$. By scaling, we may even assume that $B\subset \ZZ^d$. Therefore, in order to apply Theorem~\ref{thm:main2}, we only need to verify that $\cL_1,\cL_2$ are irreducible and coprime. For this, we now derive general conditions for irreducibility and coprimeness. We first look at irreducibility.

\begin{thm} \label{thm:irred}
$P,Q\in\Mat_d(\QQ)$ are irreducible if and only if they are invertible and the characteristic polynomial of $P^{-1}Q$ is irreducible over $\QQ$.
\end{thm}

\begin{proof}
Suppose $P,Q$ are irreducible. If $P$, say, is not invertible, then there is a one-dimensional subspace $U\subset \QQ^d$ such that $PU=0$. But then both $PU$ and $QU$ lie in the subspace $QU$ of dimension at most 1, contradicting irreducibility. 

Note that $P,Q$ are irreducible iff $R=P^{-1}Q$ has no non-trivial invariant subspace over $\QQ$. Let $p(x)\in\QQ[x]$ be the characteristic polynomial of $P^{-1}Q$. If $P^{-1}Q$ has a non-trivial invariant subspace $U$, then restricting to $U$ gives a linear transformation $R|_U:U\to U$. But the characteristic polynomial of $R|_U$ divides $p$, so $p$ is reducible.

Conversely, suppose that $p=fg$ is reducible, with $\deg f,\deg g<d$. Then at least one of $f(R),g(R)$ is not invertible, since $0=p(R)=f(R)g(R)$. 
Without loss of generality, assume that $f(R)$ is not invertible, so there is some $v\in \QQ^d-\{0\}$ such that $f(R)v=0$. If $f$ has degree $e<d$, then $R^e v$ lies in the space $U=\ang{v,Rv,\ldots,R^{e-1} v}$. Thus, $U$ is a non-trivial invariant subspace.
\end{proof}

For our coprimeness condition, we need the following lemma.

\begin{lem} \label{lem:minor}
Let $P\in\Mat_d(\QQ)$ and $Q\in \Mat_d(\ZZ)$ be such that $QP\in\Mat_d(\ZZ)$. For $1\leq k\leq d$, let $m$ be a $k\times k$ minor of $P$, i.e., the determinant of a $k\times k$ submatrix. Then $m\det(Q)\in \ZZ$.
\end{lem}

\begin{proof}
Let $m$ be the $k\times k$ minor corresponding to rows $S\subseteq [d]$ and columns $T\subseteq [d]$. Construct a matrix $R\in\Mat_d(\QQ)$ as follows: the $T$ columns of $R$ are just the $T$ columns of $P$; the $S\times T^c$ submatrix of $R$ is all zeroes; and the $S^c\times T^c$ submatrix of $R$ is the identity matrix. Then $\det R=\pm m$, so that $\det(QR)=\pm m\det Q$. But the $T$ columns of $QR$ are the $T$ columns of $QP$, which has all integer entries, and each of the other columns of $QR$ is a column of $Q$, which also has integer entries. Thus, $QR\in\Mat_d(\ZZ)$, so that $m\det Q=\pm \det(QR)\in\ZZ$.
\end{proof}

\begin{thm} \label{thm:coprime}
Suppose that $P,Q\in\Mat_d(\ZZ)$ are irreducible and $p(x)\in\QQ[x]$ is the characteristic polynomial of $P^{-1}Q$. Then $P,Q$ are coprime if and only if $|\det P|$ is the smallest positive integer $g$ such that $gp\in \ZZ[x]$.
\end{thm}

\begin{proof}
Let $g$ be the smallest positive integer such that $gp\in\ZZ[x]$. Let $R,S\in\GL_d(\QQ)$ be such that $RPS,RQS \in\Mat_d(\ZZ)$. Let $M=(RPS)^{-1}RQS=S^{-1}P^{-1}QS\in \Mat_d(\QQ)$. Then the characteristic polynomial of $M$ is again $p$. By Lemma~\ref{lem:minor} with $M$ and $RPS$ as $P$ and $Q$, if $m$ is any $k\times k$ minor of $M$, then $m\det(RPS)\in\ZZ$. Suppose $p(x)=x^d+a_{d-1}x^{d-1}+\cdots+a_0$. By looking at the expansion of $p(x)=\det(xI-M)$, we see that $a_{d-k}$ can be written as a $\ZZ$-linear combination of $k\times k$ minors of $M$. Thus, $a_{d-k}\det(RPS)\in\ZZ$ for all $k$, so $g\mid \det(RPS)=\pm |\det P|\det(RS)$. In particular, if we take both $R$ and $S$ to be the identity matrix, then $g\mid |\det P|$.

Suppose now that  $g=|\det P|$. Then this implies that $|\det(RS)|\geq 1$, so $P,Q$ are indeed coprime.
Conversely, suppose that $P,Q$ are coprime. 
Consider the rational canonical form of $P^{-1}Q$, which is a block diagonal matrix similar to $P^{-1}Q$ where each block looks like
$$\begin{pmatrix}
0 & 1 & 0 & \cdots & 0\\
0 & 0 & 1 & \cdots & 0\\
0 & 0 & 0 & \cdots & 0\\
\vdots & \vdots & \vdots & \ddots & \vdots\\
0 & 0 & 0 & \cdots  & 1\\
-c_0 & -c_1 & -c_2 & \cdots & -c_{k-1}
\end{pmatrix}.$$
The characteristic polynomial of such a block is $x^k+c_{k-1}x^{k-1}+\dots+c_0$ and the characteristic polynomial $p$ of $P^{-1}Q$ is the product of the characteristic polynomials of its blocks. But, by Theorem~\ref{thm:irred}, $p(x)=x^d+a_{d-1}x^{d-1}+\cdots+a_0$ is irreducible, so the rational canonical form of $P^{-1}Q$ consists of a single block. That is, there is some $S\in\GL_d(\QQ)$ such that 
$$S^{-1}P^{-1}QS=\begin{pmatrix}
0 & 1 & 0 & \cdots & 0\\
0 & 0 & 1 & \cdots & 0\\
0 & 0 & 0 & \cdots & 0\\
\vdots & \vdots & \vdots & \ddots & \vdots\\
0 & 0 & 0 & \cdots  & 1\\
-a_0 & -a_1 & -a_2 & \cdots & -a_{d-1}
\end{pmatrix}.$$
Let $D$ be the diagonal matrix with entries $(1,1,\ldots,1,g)$. Then $D$ and $DS^{-1}P^{-1}QS$ are integer matrices. Now set $R=DS^{-1}P^{-1}$, so that $RPS,RQS\in\Mat_d(\ZZ)$. By coprimeness, $|\det R \det S|\geq 1$. But this implies that $g/|\det P|\geq 1$, so $|\det P|\leq g$. However, from before, we have $g\mid |\det P|$, so that $|\det P|=g$, as required.
\end{proof}

Using Theorems~\ref{thm:irred} and~\ref{thm:coprime}, it is now a simple matter to verify that $\cL_1,\cL_2$ are irreducible and coprime. Thus, by Theorem~\ref{thm:main2}, we have that if $\lambda\in\RR$ is an algebraic number with minimal polynomial $p(x)=a_dx^d+\cdots+a_0\in\ZZ[x]$, where all the $a_i$ are coprime, then there are $D, \sigma>0$ such that 
$$|A+\lambda\cdot A|\geq (|\det (\cL_1)|^{1/d}+|\det(\cL_2)|^{1/d})^d|A| - D|A|^{1-\sigma}$$
holds for all finite $A\subset \QQ[\lambda]$. But, taking the rescaling of the characteristic polynomial into account, $|\det(\cL_1)| = |a_d|$ and $|\det(\cL_2)| = |a_0|$, completing the proof of Theorem~\ref{thm:alg}. Though we have not stressed the point before, it is worth noting that the same proof also goes through for $A \subset \CC$ and $\lambda \in \CC$.

\section{Concluding remarks}

\noindent
{\bf Better bounds.}
Although Theorem~\ref{thm:main} can be tight, we suspect that there is a stronger general bound. In analogy with the algebraic number setting, given $\cL\in\Mat_d(\QQ)$ with minimal polynomial $f(x)\in\ZZ[x]$, which we assume to have coprime coefficients, suppose that $f(x) = \prod_{i=1}^d (a_ix+ b_i)$ is a full complex factorisation of $f$ and let $H(\cL)=\prod_{i=1}^d(|a_i|+|b_i|)$. Our conjecture, a variant of a recent conjecture of Krachun and Petrov~\cite[Conjecture~2]{KP20} that was itself inspired by a continuous analogue~\cite[Theorem~2]{KP20}, is then as follows.

\begin{conj} \label{conj:better}
Let $\cL_1,\cL_2\in\Mat_d(\ZZ)$ be irreducible. Then, for any finite subset $A$ of $\ZZ^d$,
$$|\cL_1 A+\cL_2 A|\geq H(\cL_1^{-1}\cL_2)|A|-o(|A|).$$
\end{conj}

Note that the coprimeness condition is unnecessary here, since if we were to replace $\cL_1,\cL_2$ with $\cP \cL_1\cQ,\cP \cL_2\cQ$ to make them coprime, then
$$H((\cP \cL_1\cQ)^{-1}(\cP \cL_2\cQ))=H(\cQ^{-1}\cL_1^{-1}\cL_2\cQ)=H(\cL_1^{-1}\cL_2).$$

Moreover, Conjecture~\ref{conj:better} implies our Theorem~\ref{thm:main}, since if $\cL_1,\cL_2$ are coprime, then, by Theorem~\ref{thm:coprime}, the minimal polynomial of $\cL_1^{-1}\cL_2$ over $\ZZ$ is $c_dx^d+c_{d-1}x^{d-1}+\dots+c_0$, where $|c_d|=|\det (\cL_1)|$ and $|c_0|=|c_d||\det(\cL_1^{-1}\cL_2)|=|\det (\cL_2)|$. Therefore, by H\"older's inequality, 
\begin{align*}
    H(\cL_1^{-1}\cL_2) &= \prod_{i=1}^d (|a_i|+|b_i|)
    \geq \paren{\prod_{i=1}^d |a_i|^{1/d}+\prod_{i=1}^d |b_i|^{1/d}}^d
    = (|c_d|^{1/d}+|c_0|^{1/d})^d.
\end{align*}

There should also be a suitable generalisation of Conjecture~\ref{conj:better} to more than two variables, but, unlike Conjecture~\ref{conj:bukh2}, which itself remains open for three or more variables, it is not at all obvious what this should be. \\

\noindent
{\bf Lower-order terms.} 
Unlike with sums of dilates (see, for instance,~\cite{BS14, CL22, H21, S16}), we cannot in general hope for the error term in Theorem~\ref{thm:main} to be a constant. Indeed, in two dimensions, if we set $A = \setcond{(x,y)}{0 \leq x, y \leq n - 1}$ and $\cL$ to be the anti-clockwise rotation about  the origin through $\pi/2$, then $|A| = n^2$, but $|A + \cL A| = (2n-1)^2 = 4|A| - 4|A|^{1/2} + 1$. That is, the error term in this case is a multiple of $|A|^{1/2}$. Similarly, in $d$ dimensions, there are examples for which the error term is a multiple of $|A|^{1 - 1/d}$. Following Shakan~\cite{S16}, we conjecture that there are no significantly worse examples.

\begin{conj} \label{conj:refined}
Suppose that $\cL_1,\ldots,\cL_k\in \Mat_d(\ZZ)$ are irreducible and coprime. Then there is a constant $D$ such that, for any finite subset $A$ of $\ZZ^d$,
$$|\cL_1 A+\cdots +\cL_k A|\geq \paren{|\det(\cL_1)|^{1/d}+\cdots +|\det(\cL_k)|^{1/d}}^d|A| - D |A|^{1-1/d}.$$
\end{conj}

A proof of this conjecture when $k = 2$ would already constitute a significant improvement on our Theorem~\ref{thm:main}, which gives an error term of the form $D|A|^{1-\sigma}$ for some $\sigma > 0$ which depends not only on $d$, but also on $|\det(\cL_1)|$ and $|\det(\cL_2)|$. \\

\vspace{2mm}
\noindent
{\bf Real-valued analogues.}
Our main result, Theorem~\ref{thm:main}, can be extended to subsets of $\RR^d$ as follows.

\begin{thm} \label{thm:mainreal}
Suppose that $\cL_1,\cL_2\in \Mat_d(\ZZ)$ are irreducible and coprime. Then there are constants $D$, $\sigma>0$ such that, for any finite subset $A$ of $\RR^d$, 
$$|\cL_1 A+\cL_2 A|\geq \paren{|\det(\cL_1)|^{1/d}+|\det(\cL_2)|^{1/d}}^d|A|-D|A|^{1-\sigma}.$$
\end{thm}

To see this, suppose that $A \subset \RR^d$ and let $B \subset \RR$ be the set consisting of all real numbers that appear as a coordinate of some element of $A$. For any fixed natural number $k$, a standard result in additive combinatorics (see, for instance,~\cite[Lemma~5.25]{TV06}) allows us to find a set $B' \subset \ZZ$ which has a Freiman isomorphism of order $k$ with $B$. We then obtain a set $A' \subset \ZZ^d$ by replacing each coordinate of each element of $A$ with its image in $B'$. Provided $k$ is chosen sufficiently large in terms of the coefficients of $\cL_1$ and $\cL_2$, it is now easy to verify that $|\cL_1 A'+\cL_2 A'| = |\cL_1 A+\cL_2 A|$. Therefore, since the conclusion of the theorem is known for all $A' \subset \ZZ^d$, it is also true for all $A \subset \RR^d$.

Our results also allow us to say something about the size of $A  + \cL A$ when $\cL$  has real algebraic entries. In this case, by a process similar to that used in Section~\ref{sec:alg} to estimate $|A + \lambda \cdot A|$ for algebraic $\lambda$ and $A \subset \RR$, we can convert the problem of estimating $|A  + \cL A|$ to one of estimating $|\cL_1 B+\cL_2 B|$ for some integer matrices $\cL_1,\cL_2\in\Mat_{d'}(\ZZ)$ and $B \subset \ZZ^{d'}$. Indeed, a generalisation of Lemma~\ref{lem:alg} allows us to assume that $A\subset \QQ[\lambda_1,\lambda_2,\ldots]^d$, where $\lambda_1,\lambda_2,\ldots$ are the entries of $\cL$, so that the problem becomes equivalent to estimating $|A'+\cL' A'|$ for some $\cL'\in\Mat_{d'}(\QQ)$ and some $A'\subset \QQ^{d'}$ with $|A'| = |A|$. Clearing denominators, we can then rephrase this as the problem of estimating $|\cL_1 B+\cL_2 B|$ for some $\cL_1,\cL_2\in\Mat_{d'}(\ZZ)$ and some $B \subset \ZZ^{d'}$ with $|B| = |A|$. Finally, if $\cL_1,\cL_2$ are not irreducible or coprime, we can make them irreducible by restricting to a subspace and coprime by replacing them with a suitable $\cP\cL_1\cQ,\cP\cL_2\cQ$.

\end{document}